\documentclass[titlepage,11pt]{article}

\usepackage[margin=1.2in]{geometry}
\usepackage{latexsym}
\usepackage{amsfonts, amsthm}
\usepackage[leqno]{amsmath}

\usepackage{MnSymbol}%
\usepackage{wasysym}%

\usepackage{thmtools, subcaption}
\usepackage{thm-restate}

\makeatletter
\newcommand{\leqnomode}{\tagsleft@true}
\newcommand{\reqnomode}{\tagsleft@false}
\makeatother

\newcommand{\dichi}{\vec{\chi}}

\newenvironment{subproof}[1][\proofname]{%
  \begin{proof}[#1]%
}{%
  \end{proof}%
}
\def\longbox#1{\parbox{0.85\textwidth}{#1}}

\newcommand{\sta}[1]{\vspace*{-0.6cm}
  \begin{equation}
    \longbox{{\emph{#1}}}
\end{equation}}

\usepackage{tikz}
\usepackage{float}
\usepackage{url}
\usetikzlibrary{arrows}
\usetikzlibrary{decorations.markings}
\usepackage[symbol]{footmisc}

\usepackage{comment}

\usepackage{xspace}
\usepackage[color=Olive]{todonotes}

\usepackage[colorlinks=true,allcolors=blue]{hyperref}
\usepackage{cleveref}

\def\longbox#1{\parbox{0.85\textwidth}{#1}}

\usepackage{authblk}

\title{On heroes in digraphs with forbidden induced forests}

\author[1]{Alvaro Carbonero}
\author[1]{Hidde Koerts}
\author[2]{Benjamin Moore}
\author[1]{Sophie Spirkl\thanks{Emails: (ar2carbonerogonzales, hkoerts, sspirkl)@uwaterloo.ca, brmoore@iuuk.mff.cuni.cz \\
Spirkl: We acknowledge the support of the Natural Sciences and Engineering Research Council of Canada (NSERC), [funding reference number RGPIN-2020-03912]. Cette recherche a été financée par le Conseil de recherches en sciences naturelles et en génie du Canada (CRSNG), [numéro de référence RGPIN-2020-03912]. This project was funded in part by the Government of Ontario. \\
Benjamin Moore is supported by project 22-17398S (Flows and cycles in graphs on surfaces) of the Czech Science Foundation.\\
Some of this material appeared in Carbonero's master's thesis. 
}}

\affil[1]{University of Waterloo, Department of Combinatorics and Optimization, Waterloo, Canada}

\affil[2]{Charles University, Institute of Computer Science, Prague, Czech Republic}

\date{\today}

\newtheorem{thm}{Theorem}[section]

\newtheorem{conjecture}[thm]{Conjecture}
\newtheorem{lemma}[thm]{Lemma}

\newtheorem{question}[thm]{Question}

\begin{document}
\maketitle
\begin{abstract}
We continue a line of research which studies which hereditary families of digraphs have bounded dichromatic number. For a class of digraphs $\mathcal{C}$, a hero in $\mathcal{C}$ is any digraph $H$ such that $H$-free digraphs in $\mathcal{C}$ have bounded dichromatic number. We show that if $F$ is an oriented star of degree at least five, the only heroes for the class of $F$-free digraphs are transitive tournaments. For oriented stars $F$ of degree exactly four, we show the only heroes in $F$-free digraphs are transitive tournaments, or possibly special joins of transitive tournaments. Aboulker et al. characterized the set of heroes of $\{H, K_{1} + \vec{P_{2}}\}$-free digraphs almost completely, and we show the same characterization for the class of $\{H, rK_{1} + \vec{P_{3}}\}$-free digraphs. Lastly, we show that if we forbid two ``valid" orientations of brooms, then every transitive tournament is a hero for this class of digraphs.
\end{abstract}

\maketitle

\section{Introduction}

Throughout this paper, (di)graphs are finite and simple. In particular, for digraphs, between every two vertices $u$ and $v$, at most one of $uv$ and $vu$ is present. 

We will be interested in the dichromatic number of families of digraphs with forbidden induced subgraphs. Recall that a (di)graph $H$ is an \textit{induced subgraph} of a (di)graph $G$ if by deleting vertices of $G$ we obtain a (di)graph isomorphic to $H$. Equivalently, we say $H$ is an induced subgraph of $G$ if there exists a set $S\subseteq V(G)$ such that $G[S]$ is isomorphic to $H$. We call $S$ a \textit{copy} of $H$ in $G$. If $G$ has no copy of $H$, then we say $G$ is \textit{$H$-free}. Furthermore, if $\mathcal{C}$ is a set of (di)graphs, then $G$ is \textit{$\mathcal{C}$-free} if for every $H\in\mathcal{C}$, we have that $G$ is $H$-free. 

For a natural number $k$, we write $[k]$ for the set $\{1, \dots, k\}$. Given a digraph $D$, a \emph{$k$-dicolouring} of $D$ is a function $f : V(G) \rightarrow [k]$ such that for every $i \in [k]$, the induced subdigraph $D[f^{-1}(i)]$ is acyclic; in other words, no directed cycle of $D$ is monochromatic with respect to $f$. The \emph{dichromatic number $\dichi(D)$}, introduced by Neumann-Lara in \cite{first-dicoloring}, is the minimum $k$ such that $D$ has a $k$-dicolouring. If $X \subseteq V(D)$, we also use $\dichi(X)$ to mean $\dichi(D[X])$. One may compare this definition to the \textit{chromatic number}, denoted $\chi(G)$, which is the minimum $k$ such that there is a mapping $f:V(G) \to [k]$ where for all edges $e=xy$, $f(x) \neq f(y)$.  

We are interested in understanding for which families $\mathcal{F}$ all $\mathcal{F}$-free graphs have bounded dichromatic number. To this end, let us say that $\mathcal{F}$ is \emph{$\dichi$-finite} if there is a constant $c$ such that all $\mathcal{F}$-free digraphs have dichromatic number at most $c$. We consider the following question, which was first systematically studied by Aboulker, Charbit, and Naserasr \cite{gyarfas-sumner-digraphs}: 
\begin{question} \label{q:main} Which finite families $\mathcal{F}$ of digraphs are $\dichi$-finite?  
\end{question}

Consider the family of tournaments: that is, the family of graphs which is obtained by orienting cliques. It is easy to see that tournaments are exactly the class of digraphs which forbid $2K_{1}$, that is the graph consisting of two isolated vertices. It is also well-known that there exist tournaments of arbitrarily large dichromatic number; in other words, $\{2K_1\}$ is not $\dichi$-finite. A natural question is: for which digraphs $H$ is $\{2K_1, H\}$ a $\dichi$-finite family? This motivates the definition of heroes.

We say that $H$ is a \emph{hero} in $\mathcal{F}$-free digraphs if $\{H\} \cup \mathcal{F}$ is $\dichi$-finite. Heroes in tournaments are often just called heroes. A seminal paper of Berger et al.\ \cite{heroes-characterization} completely characterizes heroes in tournaments. We need some definitions before we can fully state the theorem.

For graphs and digraphs $D_1$ and $D_2$, we use $D_1+D_2$ to denote the disjoint union of $D_1$ and $D_2$, and we use $rD_1$ for an integer $r\geq 0$ to denote the disjoint union of $r$ copies of $D_1$.   A tournament is \emph{transitive} if it is acyclic, and a transitive tournament on $k$ vertices is denoted as $TT_k$. For two digraphs $D_1$ and $D_2$, we define $D_1 \Rightarrow D_2$ to be the digraph arising from $D_1 + D_2$ by adding all arcs $d_1d_2$ with $d_i \in V(D_i)$ for $i \in \{1, 2\}$. Furthermore, for digraphs $D_1, D_2, D_3$, we define $\Delta(D_1, D_2, D_3)$ as the digraph arising from $D_1 + D_2 + D_3$ by adding all arcs $d_id_j$ with $d_i \in V(D_i), d_j \in V(D_j)$ and $(i, j) \in \{(1, 2), (2, 3), (3, 1)\}$. For convenience, if $D_i$ is a $k$-vertex transitive tournament, we write $k$ for $D_i$ in this construction. 

Now we can state the characterization of heroes in tournaments:
\begin{thm}[Berger, Choromanski, Chudnovsky, Fox, Loebl, Scott, Seymour, and Thomassé \cite{heroes-characterization}]
\label{k_2_cooperates}
   A digraph $H$ is a hero in tournaments if and only if one of the following holds:
    \begin{itemize}
        \item $H=K_1$;
        \item $H=H_1\Rightarrow H_2$ where $H_1$ and $H_2$ are heroes in tournaments; or
        \item $H = \Delta(1, H_1, m)$ or $H = \Delta(1, m, H_1)$ where $m\geq 1$ and $H_1$ is a hero in tournaments.
    \end{itemize}
\end{thm}

While this may seem like a very special case of Question \ref{q:main}, it is particularly relevant due to a theorem of Aboulker, Charbit, and Naserasr \cite{gyarfas-sumner-digraphs}, who showed that if $\mathcal{F}$ is $\dichi$-finite, then $\mathcal{F}$ contains both an oriented forest (a digraph whose underlying undirected graph is a forest) and a hero in tournaments. They further showed: 
\begin{thm}[Aboulker, Charbit, and Naserasr \cite{gyarfas-sumner-digraphs}]
    If $\{H, F\}$ is $\dichi$-finite where $H$ is a hero in tournaments and $F$ is an oriented forest, then either $H$ is a transitive tournament, or the underlying undirected graph of $F$ is a disjoint union of stars. 
\end{thm}

What happens in the case where $H$ is a transitive tournament and $F$ is an oriented forest? Chudnovsky, Scott and Seymour \cite{stars-dichi-bounded} showed that in the special case where $F$ is an \emph{oriented star} of degree $t$, that is, an orientation of $K_{1,t}$, the set $\{H, F\}$ is $\dichi$-finite.

\begin{thm}[Chudnovsky, Scott, and Seymour \cite{stars-dichi-bounded}]
\label{stars-dichi-bounded}
    If $F$ is an oriented star and $H$ is a transitive tournament, $\{H, F\}$ is $\dichi$-finite. 
\end{thm}

We show that if the star $F$ has (undirected) degree at least five, then the only heroes in $F$-free digraphs are transitive tournaments, showing that the theorem of Chudnovsky, Scott, and Seymour cannot be strengthened for these stars. 

\begin{restatable}{thm}{mainthree}
\label{main3}
    If $F$ is an oriented star of degree at least 5, then $\{H, F\}$ is $\dichi$-finite only if $H$ is a transitive tournament.
\end{restatable}

In the case where $F$ is an orientation of a star of degree $4$, the only possible heroes are transitive tournaments and tournaments of the form $\Delta(1, m, m')$: 

\begin{restatable}{thm}{mainfive}
\label{main5}
    If $F$ is an oriented star of degree 4 and $\{H, F\}$ is $\dichi$-finite, then either $H$ is a transitive tournament or $H=\Delta(1, m, m')$ where $m, m'\geq 1$.
\end{restatable}
While the case when $H$ is a transitive tournament is resolved by Theorem \ref{stars-dichi-bounded}, the case when $H = \Delta(1, m, m')$ remains open. 

For oriented stars of degree $2$ and $3$, the full picture is not yet clear. For oriented stars of degree 3, we are not aware of any results aside from Theorem \ref{stars-dichi-bounded}.  For stars of degree $2$, which are isomorphic to $P_3$, Aboulker et al.\ \cite{complete-multipartite} obtain the following:

\begin{thm}[Aboulker, Aubian, and Charbit \cite{complete-multipartite}] \label{thm:p3}
    For every hero $H$ in tournaments, $\{H, \vec{P_3}\}$ is $\dichi$-finite. 
\end{thm}

Here $\vec{P_{3}}$ is the directed path on three vertices. Let us pause to introduce some convenient notation for orientations of paths. We use arrows $\rightarrow$ and $\leftarrow$ to denote the direction of the arcs in a path. For example, $v_1 \rightarrow v_2 \rightarrow v_3 \leftarrow v_4 \leftarrow v_5$ and $\rightarrow\rightarrow\leftarrow\leftarrow$ both denote the digraph $(\{v_1,\dots, v_5 \}, \{v_1v_2, v_2v_3, v_4v_3, v_5v_4  \})$. The \textit{directed path $\vec{P_m}$} on $m$ vertices refers to a path on $m$ vertices with orientation $\rightarrow\rightarrow\dots\rightarrow$. 

Thus returning to oriented stars, the remaining cases for oriented stars of degree 2 are when $F \in \{\rightarrow \leftarrow, \leftarrow \rightarrow\}$. These cases are the same (up to reversing all arcs), so it suffices to consider $\leftarrow \rightarrow$. Steiner gave the following partial result (where $\vec{C_3}$ is the \textit{cyclic triangle} $\vec{C_3} = (\{v_{1},v_{2},v_{3}\}, \{v_{1}v_{2},v_{2}v_{3},v_{3}v_{1}\})$). 
\begin{thm} [Steiner \cite{steiner2022coloring}]
    If $F = \ \leftarrow\rightarrow$ and $H=\vec{C_3}\Rightarrow TT_k$ for some integer $k\geq 1$, then $\{H, F\}$ is $\dichi$-finite. 
\end{thm}

Moving past oriented stars, we consider another natural question.
Which digraphs $F$ have the property that all heroes in tournaments are heroes in $F$-free digraphs? The following two results show that this holds if $F = rK_1$, and does not hold if $F$ contains $K_1 + \vec{P_2}$.
\begin{thm}[Harutyunyan, Le, Newman, and Thomassé \cite{dense-digraphs}] 
\label{dense-digraphs-main}
    For all $r \in \mathbb{N}$ and every hero $H$ in tournaments, $\{rK_1, H\}$ is $\dichi$-finite. 
\end{thm}

\begin{thm}[Aboulker, Aubian, and Charbit \cite{complete-multipartite}]
\label{about-122}
    If $F$ contains a copy of $K_1+\vec{P_2}$, then $\Delta(1, 2, \vec{C_3}), \Delta(1, \vec{C_3}, 2), \Delta(1, 2, 3)$, and $\Delta(1, 3, 2)$ are not heroes in $F$-free digraphs.
\end{thm}

 Complementing Theorem \ref{about-122}, Aboulker, Aubian, and Charbit \cite{complete-multipartite} almost completely characterize heroes in $\{K_1+\vec{P_2} \}$-free digraphs:
\begin{thm}[Aboulker, Aubian, and Charbit \cite{complete-multipartite}]
\label{complete-multipartite-main}
    The set $\{H, K_1+\vec{P_2} \}$ is $\dichi$-finite if: 
    \begin{itemize}
        \item $H=K_1$;

        \item $H=H_1\Rightarrow H_2$ where $\{H_i, K_1+\vec{P_2} \}$ is $\dichi$-finite for $i \in \{1, 2\}$; or

        \item $H=\Delta(1, 1, H_1)$ where $\{H_1, K_1+\vec{P_2} \}$ is $\dichi$-finite.
    \end{itemize}
\end{thm}

With this theorem, only the status of $\Delta(1, 2, 2)$ remains to be decided. 
This raises the natural question: For which forests $F$ is it the case that $\{F, H\}$ is $\dichi$-finite for all $H$ as in Theorem \ref{complete-multipartite-main}? In Section \ref{sec:heroes}, we show: 
\begin{restatable}{thm}{maintwo}  \label{main2}  Let $r \in \mathbb{N}$.  The set $\{H, rK_1+\vec{P_3} \}$ is $\dichi$-finite if: 
    \begin{itemize}
        \item $H=K_1$;

        \item $H=H_1\Rightarrow H_2$ where $\{H_i, rK_1+\vec{P_3} \}$ is $\dichi$-finite for $i \in \{1, 2\}$; or

        \item $H=\Delta(1, 1, H_1)$ where $\{H_1, rK_1+\vec{P_3} \}$ is $\dichi$-finite.
    \end{itemize}
\end{restatable}
Again, using Theorem \ref{about-122}, this leaves open only the status of $\Delta(1, 2, 2)$.

To motivate our final result, we recall the directed Gy\'arf\'as-Sumner conjecture, posed by Aboulker, Charbit and Naserasr \cite{gyarfas-sumner-digraphs}: 
\begin{conjecture}[Aboulker, Charbit, and Naserasr \cite{gyarfas-sumner-digraphs}]
\label{open-conjecture}
    If $F$ is a directed forest and $H$ is a transitive tournament, then $\{H, F\}$ is $\dichi$-finite. 
\end{conjecture}

As noted in \cite{gyarfas-sumner-digraphs}, this is a directed analog of the famous Gy\'arf\'as-Sumner conjecture.

\begin{conjecture}[Gy\'arf\'as \cite{gyarfas-sumner-1} and Sumner\cite{gyarfas-sumner-2}]
\label{gyarfas-sumner-conjecture}
    For every forest $F$ and every clique $K_k$ on $k$ vertices, the $\{F, K_k \}$-free graphs have bounded chromatic number.
\end{conjecture}

Conjecture \ref{open-conjecture} is wide open. We do not even know if the conjecture holds when $P$ is an oriented path. Recently, Cook et al. \cite{p4-dichi-bounded} showed:
\begin{thm}[Cook, Masařík, Pilipczuk, Reinald, and Souza \cite{p4-dichi-bounded}]
    For every $k$, the set $\{TT_k, \vec{P_4}\}$ is $\dichi$-finite. 
\end{thm}

We investigate a weakening of Conjecture \ref{open-conjecture} by forbidding two specific oriented forests, called brooms. For an integer $r\geq 1$, let the $r$\textit{-broom}, denoted by $B_r$, be the graph defined as follows:
$$
B_r:=(\{v_1, v_2, v_3, w_1,\dots, w_r\}, \{v_1v_2, v_2v_3, v_3w_1, \dots, v_3w_r  \}).
$$
\begin{figure}
    \centering
    \includegraphics[scale=0.6]{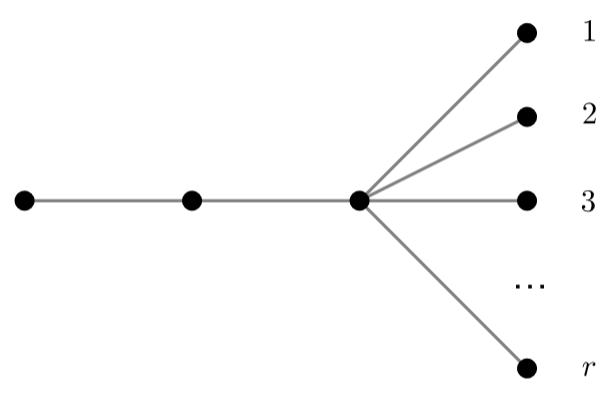}
    \caption{An illustration of $B_r$.}
    \label{intro-1}
\end{figure}
See Figure \ref{intro-1}. If $\mathcal{B}$ is an orientation of $B_r$ and $\mathcal{B'}$ is an orientation of $B_s$, then we say $\mathcal{B}$ and $\mathcal{B'}$ have \textit{opposing orientations} if $v_2v_3\in A(\mathcal{B})$ and $v_3v_2\in A(\mathcal{B'})$. Furthermore, a \textit{valid orientation} $\mathcal{B}$ of $B_r$ is an orientation such that either $\{v_3w_1,\dots,v_3w_r \}\subseteq A(\mathcal{B})$ or $\{w_1v_3, \dots, w_rv_3 \}\subseteq A(\mathcal{B})$. We prove the following, a strengthening of an unpublished result due to Linda Cook and Seokbeom Kim (private communication):

\begin{restatable}{thm}{dichimainone}
\label{dichi-main1}
    Let $r, s, t$ be positive integers. If $\mathcal{B}$ and $\mathcal{B'}$ are valid opposing orientations of $B_r$ and $B_s$ respectively, then $\{\mathcal{B}, 
    \mathcal{B'},TT_{t}\}$ is $\dichi$-finite.
\end{restatable}

We give a brief outline of how we prove our results. 

For Theorems \ref{main3} and \ref{main5}, we construct a sequence of digraphs with large dichromatic number that in the case of Theorem \ref{main3}, have no cyclic triangle, or in the case of Theorem \ref{main5}, only constrained cyclic triangles. We sketch the construction given for Theorem \ref{main3}, as the construction is very similar for Theorem \ref{main5}. Our starting point is a classical construction of a graph with no short odd cycles, and large chromatic number -- the shift graph (see, for example, \cite{chi-survey}). Shift graphs were also used by Aboulker, Aubian, and Charbit \cite{complete-multipartite} in the proof of Theorem \ref{about-122}. For Theorem \ref{main3}, we will use a $7$-tuple-shift graph. To turn this into a digraph, we simply orient in the natural fashion, that is, we orient edges from $(a,b,c,d,e,f,g) \to (b,c,d,e,f,g, \bullet)$ (in other words, edges correspond to one ``shift'' of the sequence). This results in a graph with dichromatic number $1$, so to remedy this, after four such shifts, we add a backedge (that is, adding edges of the form $(a,b,c,d,\bullet,\bullet,\bullet) \to (\bullet,\bullet,\bullet,a,b,c,d)$). By a well-known theorem of Gallai-Hasse-Roy and Vitaver, this forces the graph to have large dichromatic number. Now we need to forbid oriented stars of degree at least five, and to do this we add transitive tournaments in a careful way to the neighbourhoods of vertices. Finally, one can check that the resulting graph has large dichromatic number, no cyclic triangle, and no induced oriented star of degree at least five, completing the proof. 


For Theorem \ref{main2}, we require multiple steps. First we look at the class of ``$k$-(co)local" graphs. These are graphs that have the property that for every vertex $v$, the out-(in)-neighbourhood of $v$ induces a digraph with dichromatic number at most $k$. We show that if we are given a digraph $F$ which behaves well with respect to heroes (which we call localized and colocalized) in $k$-local $(rK_{1} + F)$-free digraphs, then we can construct new heroes from smaller heroes by the operations given in Theorem \ref{main2}. To prove this, we follow the ideas devised by Harutyunyan, Le, Newman, and Thomassé \cite{dense-digraphs} to prove Theorem \ref{dense-digraphs-main} (what they describe as characterizing superheroes), generalizing it to the setting of localized and colocalized graphs. With this in hand, to prove Theorem \ref{main2} it will suffice to show that $\vec{P_{3}}$ is localized, colocalized and has a property that we call cooperation. We will introduce a concept of ``domination" which we show implies the localization properties, and thus prove the theorem. The final step will be to prove that $\vec{P_{3}}$ has the domination property, which will then imply the theorem immediately.

For Theorem \ref{dichi-main1}, we follow a similar approach as Cook et al. \cite{p4-dichi-bounded} in their proof that $\vec{P_{4}}$-free digraphs have dichromatic number bounded by a function of their clique number. For a digraph $G$, let $\omega(G)$ denote the clique number of the underlying undirected graph of $G$. Cook et al. proceed by considering what they call a \textit{path minimizing closed tournament} and using this, they find a so called ``nice set" (we defer the definition of this until later). Nice sets are a well-known concept which first appeared in \cite{gyarfas-sumner-digraphs}, and if one can show they exist, it immediately implies Theorem \ref{dichi-main1}. We will not be able to find a nice set, but by using path minimizing closed tournaments in a similar fashion to the Cook et al. proof, we will find a slightly weaker set, which we will call a $k$-nice-set, whose existence immediately implies Theorem \ref{dichi-main1}. The majority of the difference in our result from the Cook et al. result is the additional complications that arise when trying to find a $k$-nice set rather than a nice set.

We end the introduction by outlining the structure of the paper. In Section \ref{sec:stars} we prove Theorem \ref{main3} and Theorem \ref{main5}. In Section \ref{sec:local} we build the critical tools which will lead to the prove of Theorem \ref{main2}. In Section \ref{sec:heroes} we prove Theorem \ref{main2}. In Section \ref{sec:brooms} we prove Theorem \ref{dichi-main1}.

\section{Forbidding oriented stars} \label{sec:stars}

In this section, we prove Theorem \ref{main3} and Theorem \ref{main5}. The following definitions will be needed throughout. When considering a digraph $D=(V(D), A(D))$ where $uv\in A(D)$, we say $u$ \textit{sees} $v$, and $v$ \textit{is seen by} $u$. For a digraph $D$, when we say that $X_1 \subseteq V(D)$ is complete (resp. anticomplete) to $X_2\subseteq V(D)$, we mean that this is the case for the underlying undirected graph of $D$. Additionally, $X_1$ is \textit{in-complete} (resp. \textit{out-complete}) to $X_2$ if every vertex in $X_1$ is seen by (resp. sees) every vertex in $X_2$. 

\subsection{Heroes for oriented stars of degree at least five}
\label{sec-main3}

In this subsection, we prove Theorem \ref{main3}, which we restate for the reader's convenience:

\mainthree*

To prove this theorem, as well as Theorem \ref{main5}, we need the following family of graphs. Let $n$ and $k$ be integers such that $n>2k>2$. The $k$\textit{-tuple shift-graph} with indices in $[n]$ is the graph whose vertices are of the form $(x_1, \dots, x_k)$, where $x_i\in [n]$ for every $i\in [k]$ and $x_i<x_{i+1}$ for every $i\in [k-1]$. Furthermore, two vertices $(a_1, \dots, a_k)$ and $(b_1, \dots, b_k)$ are adjacent if $a_{i+1}=b_i$ for every $i\in [k-1]$ or vice versa. In \cite{shift-graphs}, Erd\H{o}s proved the following. 

\begin{thm}[Erd\H{o}s \cite{shift-graphs}]
\label{shift-graphs}
    For every fixed $k$, if $G_n$ is the $k$-tuple shift-graph with indices in $[n]$, then $\chi(G_n)\rightarrow \infty$ as $n\rightarrow \infty$.
\end{thm}

We will also use the Gallai–Hasse–Roy–Vitaver Theorem (see \cite{gallai-roy-1, gallai-roy-2}):

\begin{thm}[Gallai-Hasse-Roy-Vitaver]
\label{gallai-roy}
    If $D$ has no directed path of length $t$ as a (not necessarily induced) subgraph, and $G$ is the underlying undirected graph of $D$, then $\chi(G)\leq t$.
\end{thm}

Like Aboulker, Aubian, and Charbit \cite{complete-multipartite} in the proof of Theorem \ref{about-122}, we orient the shift graph acyclically in the natural way, and add ``back-edges'' carefully to increase its dichromatic number. The following is the main result of this subsection:

\begin{thm}\label{counterexample-1-0}
    There exists digraphs $F_1, F_2, \dots$ such that:
    \begin{itemize}
        \item $\dichi(F_n)\rightarrow \infty$ as $n\rightarrow \infty$;

        \item for every $n\geq 1$ and $v\in V(F_n)$, the neighbourhood of $v$ can be partitioned into four tournaments; and

        \item for every $n\geq 1$, the digraph $F_n$ has no cyclic triangle $\Delta(1, 1, 1)$.
    \end{itemize}
\end{thm}

\begin{proof}
    Let $G_n$ be the 7-tuple shift-graph with indices in $[n]$, and let $D_n$ be the orientation of $G_n$ where $(a_1, \dots, a_7)(b_1,\dots,b_7)\in A(D_n)$ if $b_i=a_{i+1}$ for every $i\in [6]$. For every $v=(a_1,\dots,a_7)\in V(G_n)$, define $m(v)=a_4$. Let $X := A(D_n)$. That is, $X$ is the set of edges of the form $(\bullet, b, c, d, e, f, g)\rightarrow(b, c, d, e, f, g, \bullet)$.  Moreover, let $D_n'$ be the digraph with $V(D_n')=V(D_n)$ and $A(D_n')=X\cup Y$ where $Y$ is the set of arcs of the form $(a, b, c, d, \bullet, \bullet, \bullet)\rightarrow (\bullet, \bullet, \bullet, a, b, c, d)$. Note that as $m$ is strictly increasing along arcs in $X$, it follows that $X$ is acyclic. Likewise, $m$ is strictly decreasing in $Y$, so $Y$ is acyclic.

\sta{\label{counterexample-1-1}
For every $n\geq 1$, $\chi(G_n)/3\leq \vec{\chi}(D_n')$.
}
\begin{subproof}
    We will prove the claim by proving that a set of vertices that induces an acyclic set in $D_n'$ also induces a subgraph with chromatic number at most 3 in $G_n$. Let $\Lambda$ be a set of vertices that induces an acyclic set in $D_n'$. Notice that $D_n[\Lambda]$ does not have a directed path of length 3 because if such a path $v_1\rightarrow v_2\rightarrow v_3\rightarrow v_4$ exists, then $v_4v_1\in A(D_n')$ contradicting that $\Lambda$ is an acyclic set in $D_n'$. Thus, by Theorem \ref{gallai-roy}, we have $\chi(G_n[\Lambda])\leq 3$ as desired. 
\end{subproof}

Finally, let $F_n$ be the digraph with $V(F_n)=V(D_n')$ and $A(F_n) = X\cup Y \cup Z_1\cup Z_2$ where we define $Z_1$ and $Z_2$ as follows. Let $<$ be a total ordering of $V(F_n)$. Define $Z_1$ (resp. $Z_2$) as the set of edges such that $uv\in Z_1$ (resp. $uv\in Z_2$) if $u<v$ and there exists numbers $a, b, c, d$ (resp. $d, e, f, g$) such that both $u$ and $v$ are of the form $(a, b, c, d, \bullet, \bullet, \bullet)$ (resp. $(\bullet, \bullet, \bullet, d, e, f, g)$).

\sta{\label{counterexample-1-2}
If $v\in V(F_n)$, then $N_{F_n}(v)$ can be partitioned into four tournaments.
}
\begin{subproof}
    Fix $v = (a, b, c, d, e, f, g)\in V(F_n)$. The neighbours $u$ of $v$ such that $uv\in X$ or $vu\in X$ are of the form $(\bullet, a, b, c, d, e, f)$ and $(b, c, d, e, f, g, \bullet)$ respectively. By the definition of edges in $Z_1\cup Z_2$, vertices of these forms each induce a tournament. 

    Denote by $A$ and $B$ the neighbours of $v$ of the forms $(\bullet, \bullet, \bullet, a, b, c, d)$ and $(d, e, f, g, \bullet, \bullet, \bullet)$, respectively. Notice that these sets partition the neighbours of $v$ connected to $v$ via edges in $Y$. Denote by $M$ and $N$ the neighbours of $v$ of the form $(a, b, c, d, \bullet, \bullet, \bullet)$ and $(\bullet, \bullet, \bullet, d, e, f, g)$, respectively. Notice that these sets partition the neighbours of $v$ connected to $v$ via edges in $Z_1\cup Z_2$. By the definition of edges in $Z_1\cup Z_2$, each of the sets $A, B, M$ and $N$ induces a tournament. Furthermore, $M$ is complete to $A$, and $B$ is complete to $N$ via edges in $Y$. Since $A(F_n) = X\cup Y\cup (Z_1\cup Z_2)$, these are all the neighbours of $v$, thus finishing the proof.
\end{subproof}

\sta{\label{counterexample-1-3}
$F_n$ has no cyclic triangle.
}
\begin{subproof}
    Assume for a contradiction that there exist vertices $u, v, w$ such that $u$ sees $v$, $v$ sees $w$, $w$ sees $u$, and where $u=(a, b, c, d, e, f, g)$. 
    
    We claim that no edge in the cyclic triangle is in $Z_1\cup Z_2$. For a contradiction, assume without loss of generality that $uv\in Z_1\cup Z_2$, so $m(v) = d$. Assume first that $wu\in Z_1\cup Z_2$. Consequently, $m(w) = d$ as well, so $vw\in Z_1\cup Z_2$, which contradicts that the edges in $Z_1\cup Z_2$ form an acyclic orientation. Therefore, $wu\not\in Z_1\cup Z_2$. Assume next that $wu\in X$. Consequently, $m(w) = c$, so $vw\in X$. Thus, $v = (\bullet, \bullet, a, b, c, d, e)$, which contradicts that $m(v) = d$. Therefore, $wu\not\in X$. Thus, $wu\in Y$, so $w = (d, e, f, g, \bullet, \bullet, \bullet)$. But then $vw\not\in X\cup (Z_1\cup Z_2)$, so $vw\in Y$. Thus, the first index of $v$ is $g$. This contradicts that $m(v) = d$ since $d<g$. We conclude that no edge in the cyclic triangle is in $Z_1\cup Z_2$.

    We claim that no edge in the cyclic triangle is in $X$. For a contradiction, assume without loss of generality that $uv\in X$, so $v = (b, c, d, e, f, g, \bullet)$. Assume $vw\in X$. Consequently, $w = (c, d, e, f, g, \bullet, \bullet)$, so by definition $wu\not\in Y$. But $wu\not\in X$ since $m(u)\not = g$, which contradicts that $wu$ is an arc and $wu\not\in Z_1\cup Z_2$. Thus, $vw\in Y$, so $w =(\bullet, \bullet, \bullet, b, c, d, e)$. But then $wu\not\in X$ and $wu\not\in Y$. This contradicts that $wu$ is an arc and $wu\not\in Z_1\cup Z_2$. We conclude that no edge in the cyclic triangle is in $X$. But then every edge in the cyclic triangle is in $Y$, which contradicts that the edges in $Y$ induce an acyclic digraph. This finishes the proof.
\end{subproof}

The second and third bullet points are proven in (\ref{counterexample-1-2}) and (\ref{counterexample-1-3}) respectively. Since $F_n$ contains $D_n'$ as a subgraph, it follows that $\chi(G_n)/3 \leq \vec{\chi}(F_n')$ as well. As mentioned, the sequence $\chi(G_n)\rightarrow \infty$ as $n\rightarrow \infty$, so $\dichi(F_n)\rightarrow \infty$ as $n\rightarrow \infty$ as well. Thus, the first bullet point holds.
\end{proof}

\noindent\textbf{Proof of Theorem \ref{main3}: }Assume $F$ is a directed star of degree at least 5. By Theorem \ref{stars-dichi-bounded}, every transitive tournament is a hero in $F$-free digraphs. For the other direction, assume that $H$ is a hero in $F$-free digraphs. If $H$ is not transitive, then $H$ contains a cyclic triangle. Thus, the cyclic triangle is a hero in $F$-free digraphs. This, however, contradicts  Theorem \ref{counterexample-1-0} which provides a family of digraphs of arbitrarily high dichromatic number with no cyclic triangles and which is $F$-free (the construction is $F$-free because a copy of $F$ contains a vertex whose neighbourhood has a stable set with at least 5 vertices, contradicting that the neighbourhood of every vertex can be partitioned into four tournaments). Thus, we conclude that $H$ is transitive, which finishes the proof. \qed

\subsection{Heroes for oriented stars of degree 4}
\label{sec-main4}

In this section, we prove Theorem \ref{main5}, which we restate for the reader's convenience:

\mainfive*

The proof is similar to the proof of Theorem \ref{main3}. We start by first restricting some of the heroes in $F$-free digraphs when $F$ is an oriented star of degree 4. Let the \textit{in-triangle}, denoted by $IT$, be the digraph on 4 vertices $a, b, c, d$ where $d$ is in-complete from $a, b, c$ and where $\{a, b, c\}$ induces the cyclic triangle. The first step towards proving Theorem \ref{main5} is proving the following.

\begin{thm}\label{main4}
    If $ST$ is a directed star of degree at least 4, then no hero in $ST$-free digraphs contains the in-triangle as a subgraph.
\end{thm}

This theorem is an immediate consequence to the following theorem.

\begin{thm}
    There exists digraphs $F_1, F_2, \dots$ such that:
    \begin{itemize}
        \item $\dichi(F_n)\rightarrow \infty$ as $n\rightarrow \infty$;

        \item for every $n\geq 1$ and $v\in V(F_n)$, the neighbourhood of $v$ can be partitioned into three tournaments; and

        \item for every $n\geq 1$, the digraph $IT$ is not a subgraph of $F_n$.
    \end{itemize}
\end{thm}

\begin{proof}
Let $G_n$ be the 5-tuple shift-graph with indices in $[n]$, and let $D_n$ be the orientation of $G_n$ where $(a_1,\dots,a_5)(b_1,\dots,b_5)\in A(D)$ if $b_i=a_{i+1}$ for every $i\in [4]$. For every $v=(a_1,\dots,a_5)\in V(G_n)$, define $m(v)=a_3$. Let $X = A(D_n)$. That is, $X$ is the set of edges of the form $(\bullet, b, c, d, e)\rightarrow (b, c, d, e, \bullet)$. Moreover, let $D_n'$ be the digraph with $V(D_n') = V(D_n)$ and $A(D_n') = X\cup Y$ where $Y$ is the set of arcs of the form $(a, b, c, \bullet, \bullet)\rightarrow (\bullet, \bullet, a, b, c)$. Note that as $m$ is strictly increasing along arcs in $X$, it follows that $X$ is acyclic. Likewise, $m$ is strictly decreasing in $Y$, so $Y$ is acyclic.

\sta{\label{counterexample-1}
For every $n\geq 1$, we have $\chi(G_n)/2\leq \vec{\chi}(D_n')$.
}
\begin{subproof}
    We will prove the claim by proving that a set of vertices that induces an acyclic set in $D_n'$ also induces a bipartite subgraph in $G_n$. Let $\Lambda$ be a set of vertices that induces an acyclic set in $D_n'$. Notice that $D_n[\Lambda]$ does not have a directed path of length 3 because if such a path $v_1\rightarrow v_2\rightarrow v_3$ exists, then $v_3v_1\in A(D_n')$ contradicting that $\Lambda$ is an acyclic set in $D_n'$. Thus, by Theorem \ref{gallai-roy}, we have $\chi(G_n[\Lambda])\leq 2$ as desired. 
\end{subproof}

Finally, let $F_n$ be the digraph with $V(F_n)=V(D_n')$ and $A(F_n) = X\cup Y \cup Z_1\cup Z_2$ where we define $Z_1$ and $Z_2$ as follows. Let $<$ be a complete ordering of $V(F_n)$. Define $Z_1$ (resp. $Z_2$) as the set of edges such that $uv\in Z_1$ (resp. $uv\in Z_2$) if $u<v$ and there exists numbers $a, b, c$ (resp. $c, d, e$) such that both $u$ and $v$ are of the form $(a, b, c, \bullet, \bullet)$ (resp. $(\bullet, \bullet, c, d, e$).

\sta{\label{counterexample-2}
If $v\in V(F_n)$, then $N_{F_n}(v)$ can be partitioned into three tournaments.
}
\begin{subproof}
    Fix $v = (a, b, c, d, e)\in V(F_n)$. The neighbours $u$ of $v$ such that $uv\in X$ or $vu\in X$ are of the form $(\bullet, a, b, c, d)$ and $(b, c, d, e, \bullet)$. Vertices of the former type are complete to the vertices of the latter type by edges in $Y$. Thus, vertices adjacent to $v$ via an edge in $X$ form a clique.

    Denote by $A$ and $B$ the neighbours of $v$ of the forms $(\bullet, \bullet, a, b, c)$ and $(c, d, e, \bullet, \bullet)$, respectively. Notice that these sets partition the neighbours of $v$ connected to $v$ via edges in $Y$. Denote by $M$ and $N$ the neighbours of $v$ of the form $(a, b, c, \bullet, \bullet)$  and $(\bullet, \bullet, c, d, e)$, respectively. Notice that these sets partition the neighbours of $v$ connected to $v$ via edges in $Z_1\cup Z_2$. By the definition of edges in $Z_1\cup Z_2$, each of the sets $A, B, M$ and $N$ induce a tournament. Furthermore, $M$ is complete to $A$, and $B$ is complete to $N$ via edges in $Y$. Since $A(F_n) = X\cup Y\cup (Z_1\cup Z_2)$, these are all the neighbours of $v$, thus finishing the proof. Figure \ref{figures-heroes-3} illustrates the neighbourhood of a vertex.
\end{subproof}

\begin{figure}
    \centering
    \includegraphics[width=0.8\textwidth]{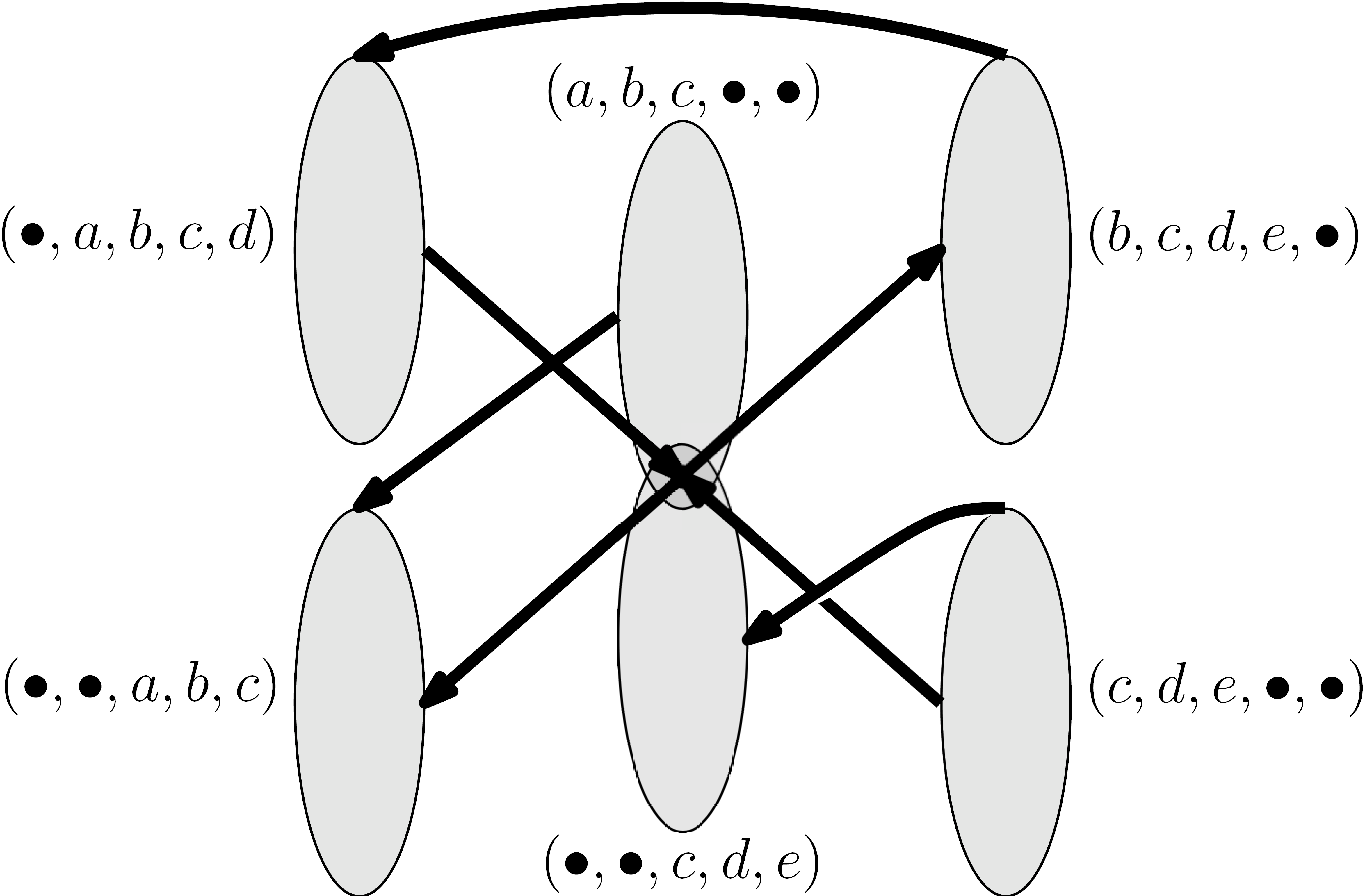}
    \caption{Illustration of the neighbourhood of the vertex $(a, b, c, d, e)$ in $F_n$.}
    \label{figures-heroes-3}
\end{figure}

\sta{\label{counterexample-3}
Every cyclic triangle in $F_n$ has two edges in $X$ and one edge in $Y$.
}
\begin{subproof}
    Let $u, v$, and $w$ be vertices such that $u$ sees $v$, $v$ sees $w$, and $w$ sees $u$. For a contradiction, assume that $uv\in Z_1\cup Z_2$, and set $v = (a, b, c, d, e)$. Since $uv\in Z_1\cup Z_2$, we have $m(u) = c$. If $vw\in X$, then $m(w) = d$. Since $m(u) < m(w)$, it follows that $wu\in Y$, so $m(u) = b$, a contradiction. If $vw\in Z_1\cup Z_2$, then $m(w) = c$, contradicting the fact that edges in $Z_1\cup Z_2$ induce an acyclic graph. Thus, $vw\in Y$, so $m(w) = a$. Since $m(w) < m(u)$, it must be that $wu\in X$, so $m(u)=b$, a contradiction. We conclude that every edge  in the directed triangle is not in $Z_1\cup Z_2$. Since each of $X$ and $Y$ span acyclic graphs, we may assume $uv\in X$ and $vw\in Y$. Consequently, $m(u) = b$ and $m(w) = a$, so $m(w) < m(u)$. This implies that $wu\in X$, proving that directed triangles have two edges in $X$ and one edge in $Y$.
\end{subproof}

The second bullet point is true by (\ref{counterexample-1}). Since $F_n$ contains $D_n'$ as a subgraph, by (\ref{counterexample-2}), we have $\chi(G_n)/2\leq \vec{\chi}(F_n')$ as well. As mentioned before, we have $\chi(G_n)\rightarrow \infty$ as $n\rightarrow \infty$. Thus, $\dichi(F_n)\rightarrow \infty$. This proves the first bullet point. We prove the third bullet point by contradiction. Assume $IT$ is a subgraph of $F_n$. Let $u\rightarrow v\rightarrow w\rightarrow u$ be the directed cycle in $F_n$ and $x$ be the vertex in-complete from $\{u, v, w\}$. Without loss of generality, by (\ref{counterexample-3}), we may assume that $uv, vw\in X$. Set $v = (a, b, c, d, e)$. If $ux\in X$, then $m(x) = c$, and since $m(w) = d > m(x)$, it follows that $wx\in Y$. This implies that $m(x) = b$, a contradiction to $m(x) =c$. If $ux\in Y$, then $m(x)<a$, so $m(x) < m(v)$. It follows that $vx\in Y$, implying $m(x) = a$, a contradiction. Thus, $ux\in Z_1\cup Z_2$, so $m(x) = b$. Since $m(x)<m(v)$, it follows that $vx\in Y$, so $m(x) =a$, a contradiction. This shows that $F_n$ is $IT$-free, which finishes the proof.
\end{proof}

\noindent\textbf{Proof of Theorem \ref{main5}: } Let the out-triangle, denoted by $OT$, be $IT$ with arcs reversed. 

\sta{
$OT$ is not a hero in $ST$-free digraphs.
}
\begin{subproof}
    Let $ST'$ be $ST$ with arcs reversed. By Theorem \ref{main4}, the $\{ST', IT\}$-free digraphs do not have bounded dichromatic number. By reversing arcs, we get that $\{ST, OT\}$-free digraphs do not have bounded dichromatic number. 
\end{subproof}

Assume for a contradiction that there exists a hero $H$ such that $H$ is not acyclic and $H \not = \Delta(1, m, m')$ for integers $m, m'\geq 1$. If $H$ is not strongly connected, then there exists non-empty tournaments $H_1\subseteq H$ and $H_2\subseteq H$ such that $V(H_1)\cap V(H_2)=\emptyset$, $H_1$ is not transitive, and either $H_1$ is out-complete to $H_2$, or $H_2$ is out-complete to $H_1$. Since $H_1$ is not transitive, it contains a directed triangle $T$. If $T$ is out-complete to $H_1$, then $H$ contains a copy of $IT$, a contradiction. Thus, $H_2$ is out-complete to $T$, but then $H$ contains a copy of $OT$, a contradiction. We conclude $H$ is strongly connected.

Since $H$ is strongly connected, by Theorem \ref{k_2_cooperates}, it follows that $H=\Delta (1, m, H')$ or $H=\Delta (1, H', m)$ where $H'$ is a hero in $ST$-free digraphs. It is then enough to prove that $H'$ is acyclic. Suppose not. That is, assume that $H'$ contains a directed triangle $T$. In either case, by the structure of strongly connected heroes, $H$ contains a copy of $IT$, a contradiction. \qed

\section{Localized and colocalized digraphs} \label{sec:local}

In this section, we introduce the concept of localized and colocalized digraphs and how these conditions relate to Theorem \ref{main2}. To elaborate, we need some definitions.

In a digraph $D$, we say the \textit{out-neighbourhood} (resp. \textit{in-neighbourhood}) of a set of vertices $S\subseteq V(D)$, denoted by $N^+(S)$ (resp. $N^-(S)$), is the set of vertices not in $S$ that vertices $v\in S$ see (resp. $v\in S$ is seen by). The neighbourhood of $S$ is $N(S):=N^+(S)\cup N^-(S)$.  When $S=\{v\}$, we use $N(v), N^+(v)$, and $N^-(v)$ to denote $N(S), N^+(S)$, and $N^-(S)$ respectively. Similarly, the \textit{non-neighbourhood} of a set $S$, denoted $N^{0}(S)$, is the set of vertices not in $S$ or the neighbourhood of $S$. If $S = \{v\}$, we let $N^0(v)$ be the set of non-neighbours of $v$. So we have:
\begin{align*}
    N^+(S)&= \bigcup_{s \in S} N^+(s) \setminus S;\\
    N^-(S)&= \bigcup_{s \in S} N^-(s) \setminus S;\\
    N^0(S)&= \bigcap_{s \in S} N^0(s) \setminus S.\\
\end{align*}

A digraph $D$ is \textit{$k$-local} if, for every $v\in V(D)$, we have $\dichi(N^+(v))\leq k$. Furthermore, it is $k$\textit{-colocal} if, for every $v\in V(D)$, we have $\dichi(N^-(v))\leq k$. The concept of $k$-local digraphs was introduced by Harutyunyan, Le, Newman, and Thomassé \cite{dense-digraphs}. A digraph $F$ \textit{cooperates} if $H$ is a hero in $F$-free digraphs when one of the following three conditions hold:
\begin{itemize}
    \item $H=K_1$;
    \item $H=H_1\Rightarrow H_2$ where $H_1$ and $H_2$ are heroes in $F$-free digraphs; or
    \item $H=\Delta(1, 1, H_1)$ where $H_1$ is a hero in $F$-free digraphs.
\end{itemize}
In other words, a digraph $F$ cooperates when their heroes can be used to construct bigger heroes by using the operations described above. Notice that Theorem \ref{complete-multipartite-main} is equivalent to proving $K_1+K_2$ cooperates, and Theorem \ref{main2} is equivalent to proving that $rK_1+\vec{P_{3}}$ cooperates for every $r\geq 1$. Of the three conditions above, the first is true for every $F$. For the second, we require a result of Aboulker, Aubian and Charbit:

\begin{thm}[Aboulker, Aubian, and Charbit \cite{complete-multipartite}]
\label{cmp-mult}
    Let $H_1$, $H_2$ and $F$ be digraphs such that $H_1\Rightarrow H_2$ is a hero in $F$-free digraphs, and $H_1$ and $H_2$ are heroes in $\{K_1+F\}$-free digraphs. Then $H_1\Rightarrow H_2$ is a hero in $\{K_1+F\}$-free digraphs.
\end{thm}
Theorem \ref{cmp-mult} shows that if the second condition above holds for $F$, then it also holds for $K_1+F$. Therefore, the main goal of this section is to develop sufficient conditions for ``lifting'' the third condition from $F$ to $K_1+F$. 

We require the following technical definitions. We say a digraph $F$ is \textit{localized} (resp. \textit{colocalized}) if for every $r\geq 1$, the following two:\begin{itemize}
    \item $\Delta (1, 1, H)$ is a hero in $\{(r-1)K_1+F\}$-free digraphs; and
    \item $H$ is a hero in $\{rK_1+F\}$-free digraphs
\end{itemize}
imply that, for every fixed $k\geq 1$, $\Delta (1, 1, H)$ is a hero in $k$-local (resp. colocal) $\{rK_1+F\}$-free digraphs. These definitions are meant to describe the properties a digraph $F$ needs to have for the proof strategy of Theorem \ref{dense-digraphs-main} to apply to $F$, where in Theorem \ref{dense-digraphs-main}, $F=K_1$. 

The first step to proving Theorem \ref{main2} is proving the following.

\begin{thm}\label{main1}
    Let $F$ be a localized and colocalized digraph. If $F$ cooperates, then $rK_1+F$ cooperates for every $r\geq 0$.
\end{thm}
Aboulker, Aubian, and Charbit \cite{complete-multipartite}, in Question 5.4, ask the following. If $H$ is a hero in $\{K_1+F\}$-free digraphs and $\Delta(1, 1, H)$ is a hero in $F$-free digraphs, does it follow that $\Delta(1, 1, H)$ is a hero in $\{K_1+F\}$-free digraphs? Theorem \ref{main1} proves that the question has an affirmative answer if $F$ is localized and colocalized. 

The following lemma simplifies using Theorem \ref{main1} for certain cases.

\begin{lemma}\label{fix}
    Let $F'$ be a digraph isomorphic to $F$ with every arc reversed. If $F$ is localized, then $F'$ is colocalized.
\end{lemma}
\begin{proof}
    Assume that $\Delta(1, 1, H)$ is a hero in $\{(r-1)K_1+F'\}$-free digraphs, and assume that $H$ is a hero in $\{rK_1+F'\}$-free digraphs. Let $H'$ denote the digraph obtained from $H$ by reversing all arcs. It follows that $\Delta(1, 1, H')$ is a hero in $\{(r-1)K_1+F\}$-free digraphs, and that $H'$ is a hero in $\{rK_1+F\}$-free digraphs.
    
    Fix $k\geq 1$. We want to prove that $\Delta(1, 1, H)$ is a hero in $k$-colocal $\{rK_1+F'\}$-free digraphs. Let $D$ be a $k$-colocal $\{rK_1+F', \Delta(1, 1, H)\}$-free digraph. Let $D'$ be $D$ with every arc reversed, and notice that $D'$ is $k$-local  and $\{rK_1+F, \Delta(1, 1, H')\}$-free. Since $F$ is localized, $\Delta(1, 1, H')$ is a hero in $k$-local $\{rK_1+F \}$-free digraphs. Thus, there exists an integer $c$ (depending only on $r, k, H, F$) such that $\dichi(D')\leq c$ and hence $\dichi(D)\leq c$. Therefore, $\Delta(1, 1, H)$ is a hero in $\{rK_1+F'\}$-free digraphs. This finishes the proof. 
\end{proof}

In this section, we prove Theorem \ref{main1}. The following lemma reduces the task of proving Theorem \ref{main1} to proving the case where $r = 1$.

\begin{lemma}\label{local}
    If $F$ is localized, then $K_1+F$ is localized. Similarly, if $F$ is colocalized, then $K_1+F$ is colocalized.
\end{lemma}
\begin{proof}
This is immediate from the definition of (co)localized graphs. 
    \end{proof}

 We will use the proof strategy devised by Harutyunyan, Le, Newman, and Thomassé \cite{dense-digraphs} to prove that if $H$ is a hero in tournaments, then $\Delta(1, m, H)$, where $m\geq 1$, is a hero in $rK_1$-free digraphs, for $r\geq 2$. 

Their strategy relies on the analysis of bag chains. A $\beta$\textit{-bag} is a subset $B$ of $V(D)$ such that $\Vec{\chi}(B) = \beta$, and a $(c, \beta)$\textit{-bag-chain} is a sequence of $\beta$-bags $B_1, \dots, B_t$ such that for every $1\leq i\leq t$ and $v\in B_i$, we have:
\begin{itemize}
    \item $\Vec{\chi}(N^+(v)\cap B_{i-1})\leq c$, and
    \item $\Vec{\chi}(N^-(v)\cap B_{i+1})\leq c$.
\end{itemize}
The \textit{length} of the $(c, \beta)$-bag-chain is $t$.

As a brief outline for the upcoming proof, for a $\{\Delta(1, 1, H), rK_1+F\}$-free digraph $D$, we want to prove the following:
\begin{enumerate}
    \item For some choice of an integer $c$, and for every $\beta\in \mathbb{N}$, the absence of a $(c, \beta)$-bag-chain of length 8 implies that the digraph has bounded dichromatic number.

    \item For some choice of $c$ and $\beta'$, $(c, \beta')$-bag-chains have a bounded dichromatic number.

    \item For some choice of $c$ and $\beta'$, if there is a $(c, \beta')$-bag-chain, then vertices not in a maximal $(c, \beta')$-bag-chain have bounded dichromatic number as well.
\end{enumerate}

We will need the following Lemma:

\begin{lemma}[Aboulker, Aubian, and Charbit \cite{complete-multipartite}]
\label{not-ours}
    Let $D$ be a digraph and let $(X_1, \dots, X_n)$ be a partition of $V(D)$. Suppose that $k$ is an integer such that:
    \begin{itemize}
        \item for every $1\leq i\leq n$, we have $\Vec{\chi}(X_i)\leq k$, and

        \item for every $1\leq i<j\leq n$, if there is an arc $uv$ with $u\in X_j$ and $v\in X_i$, then $\Vec{\chi}(X_{i+1}\cup \cdots X_{j})\leq k$.
    \end{itemize}
    Then $\Vec{\chi}(D)\leq 2k$.
\end{lemma}

We need a generalization of Lemma 3.8 in \cite{complete-multipartite}, which in turn is an adaptation of 4.4 in \cite{heroes-characterization}. Our proof differs from theirs only slightly. 

\begin{lemma}\label{partition-dichi}
    Assume that there exists an integer $m$ such that:
    \begin{itemize}
        \item $\{\Delta(1, 1, H), F \}$-free digraphs $D$ have $\Vec{\chi}(D)\leq m$; and

        \item $\{H, K_1+F\}$-free digraphs $D$ have $\Vec{\chi}(D)\leq m$.
    \end{itemize}
    If $D$ is a $\{\Delta(1, 1, H), K_1+F\}$-free digraph with a partition $(X_1, \dots, X_n)$ of $V(D)$, and $m'$ an integer such that:
    \begin{itemize}
        \item for every $1\leq i\leq n$, we have $\Vec{\chi}(X_i)\leq m'$;

        \item for every $1\leq i\leq n$ and for every $v\in X_i$, we have $\Vec{\chi}(N^+(v)\cap (X_1\cup \cdots \cup X_{i-1}))\leq m'$; and

        \item for every $1\leq i\leq n$ and for every $v\in X_i$, we have $\Vec{\chi}(N^-(v)\cap (X_{i+1}\cup \cdots \cup X_{n}))\leq m'$;
    \end{itemize}
    then $\Vec{\chi}(D)\leq 6(m+m')+2$.
\end{lemma}

\begin{proof}
    We start with the following claim.

    \sta{\label{non-nbrs}
    $\dichi(N^0(v))\leq m$ for every $v\in D$.
    }

\begin{subproof}
    Since $D$ is $K_1+F$-free, it follows that $N^0(v)$ is $F$-free. Furthermore, since $D$ is $\Delta(1, 1, H)$-free, we have by the definition of $m$ that $\dichi(N^0(v))\leq m$. 
\end{subproof}
    
    Set $k'=2(m+m')+m+1$. It suffices to show that the partition $(X_1, \dots, X_n)$ satisfies the hypothesis of Lemma \ref{not-ours} with $k=k'+m'$. 
    Let $uv$ be an edge such that $u\in X_j$, $v\in X_i$, and $i<j$, and set $X = X_{i+1}\cup \cdots \cup X_{j-1}$. For a contradiction, assume that $\Vec{\chi}(X)>k'$. Let $A = (N^-(v)\cup N^0(v))\cap X$. By the hypothesis and by (\ref{non-nbrs}), the dichromatic number of $A$ is at most $m+m'$. Similarly, the set $B = (N^+(u)\cup N^0(u))\cap X$ has dichromatic number at most $m+m'$. Thus, the set $X'=X\setminus (A\cup B)$ has $\Vec{\chi}(X')>k'-2(m+m')>m$. Consequently, there exists a copy $X''$ of $H$ in $X'$. But then, by the definitions of $A$ and $B$, it follows that $\{u, v \}\cup X''$ induces a copy of $\Delta(1, 1, H)$, a contradiction. Thus, $\Vec{\chi}(X)\leq k'$, so $\Vec{\chi}(X\cup X_j)\leq k'+m'$, as desired. 
\end{proof}

We dedicate the rest of the section to proving Theorem \ref{main1}.

\noindent\textbf{Proof of Theorem \ref{main1}.} By Lemma \ref{local}, it is enough to prove the result for $r = 1$. That is, we want to prove that $K_1+F$ cooperates. Evidently, $H=K_1$ is a hero in every class of graphs. Assume then that $H_1$ and $H_2$ are heroes in $\{K_1+F\}$-free digraphs. Consequently, they are heroes in $F$-free digraphs, and since $F$ cooperates, it follows that $H_1\Rightarrow H_2$ is a hero in $F$-free digraphs. Thus, by Theorem \ref{cmp-mult}, it follows that $H_1\Rightarrow H_2$ is a hero in $\{K_1+F\}$-free digraphs.

It remains to show that $\Delta(1, 1, H)$ is a hero in $\{K_1+F\}$-free digraphs whenever $H$ is a hero in $\{K_1+F\}$-free digraphs. However, since we exclude $\Delta(1, 1, H)$ instead of $\Delta(1, k, H)$, in some places we are able to simiplify the proofs.

Let us assume that $H$ is a hero in $\{K_1+F\}$-free digraphs. Let $c$ be an integer such that $\{K_1+F, H \}$-free digraphs $D$ have $\Vec{\chi}(D)\leq c$. Since $H$ is a hero in $K_1+F$-free digraphs, $H$ is a hero in $F$-free digraphs as well, and since $F$ cooperates, it follows that $\Delta(1, 1, H)$ is a hero in $F$-free digraphs. Let $b'$ be such that $\{F, \Delta(1, 1, H) \}$-free digraphs $D$ have $\dichi(D)\leq b'$. Since $F$ is localized, set $f_1(r, k, H)$ as the function such that $\{\Delta(1, 1, H), rK_1+F \}$ $k$-local digraphs $D$ have $\dichi(D)\leq f_1(r, k, H)$ whenever $\Delta(1, 1, H)$ is a hero in $\{(r-1)K_1+F\}$-free digraphs, and $H$ is a hero in $\{rK_1+F\}$-free digraphs. Let $f_2(r, k, H)$ be the equivalent but from the fact that $F$ is colocalized. Now let $f(r, k, H) = \max\{f_1(r, k, H), f_2(r, k, H) \}$. Set
$$
\hat{f}(\beta):=2f\left(1, 2f\left(1, 2f\left(1, \beta, H\right)+1, H\right)+1, H\right).
$$
Finally, set $\beta'=2|V(H)|(c+b')+b'+1$. We will show that $\{\Delta(1, 1, H), K_1+F\}$-free digraphs $D$ have $\dichi(D)\leq b$ where 
$$
b = 6(\max\{b', c \}+\beta') +3\hat{f}(\beta')+2.
$$

Assume that $D$ is a $\{\Delta(1, 1, H), K_1+F \}$-free digraph. Henceforth, we will use the terms $\beta$\textit{-bags} and $\beta$\textit{-bag-chains} to refer to $(c, \beta)$-bags and $(c, \beta)$-bag-chains. To achieve the first objective, we start by proving that the absence of a $\beta$-bag-chain of length 2 bounds the dichromatic number. Call a vertex $v$ $\beta$\textit{-red} if $\dichi(N^+(v))\leq \beta$, and $\beta$\textit{-blue} if $\dichi(N^-(v))\leq \beta$. The following two claims are the equivalent of Lemma 4.11 in \cite{dense-digraphs}, although our proof is significantly simpler as we deal with $\Delta(1, 1, H)$ instead of $\Delta(1, m, H)$ for some $m$.

\vspace{2mm}
\sta{\label{balls1}
For every $\beta\in \mathbb{N}$, if $D$ does not have a $\beta$-bag-chain of length 2, then $\dichi(D)\leq 2f(1, \beta, H)$.
}
\begin{subproof}
    Set $R, B$, and $U$ as the sets of $\beta$-red, $\beta$-blue and uncoloured vertices respectively. We start by proving that $U$ is empty. For the sake of a contradiction, assume that $u\in U$. Set $B_1 = N^-(u)$ and $B_2 = N^+(u)$. We claim that $B_1, B_2$ is a $\beta$-bag-chain. Let $v\in B_1$.  If $\dichi(N^-(v)\cap B_2)>c$, then there exists a copy $X$ of $H$ in $B_2$. But then $\{u, v \} \cup X$ induces a copy of $\Delta(1, 1, H)$ in $D$, a contradiction. A symmetric argument proves that if $v\in B_2$, then $\dichi(N^+(v)\cap B_1)\leq c$. That is, $B_1, B_2$ is $\beta$-bag-chain of length 2, a contradiction. Thus, $U$ is empty. Notice that $D[R]$ is $d$-local, so $\dichi(R)\leq f(1, d, H)$. Similarly, $D[B]$ is $\beta$-colocal, so $\dichi(B)\leq f(1, \beta, H)$, and hence $\dichi(D)\leq 2f(1, \beta, H)$, as claimed.
\end{subproof}

\sta{\label{balls2}
For every $\beta\in \mathbb{N}$, if $D$ does not have a $\beta$-bag-chain of length 8, then  $\dichi(D)\leq \hat{f}(\beta)$.
}
\begin{subproof}
    We proceed by contrapositive. Assume that $\dichi(D)>\hat{f}$. By (\ref{balls1}), there exists a $\left(2f\left(1, 2f\left(1, \beta, H\right)+1, H\right)+1\right)$-bag-chain of length 2, say $A_1, A_2$. By definition of a bag and by (\ref{balls1}), it follows that $A_1$ contains a $(2f(1, \beta, H)+1)$-bag-chain of length 2 consisting of bags $A_1^1, A_1^2$. Similarly, $A_2$ contains the $(2f(1, \beta, H)+1)$-bag-chain $A_2^1, A_2^2$. Finally, using the same reasoning, we can split each of these bags into the $\beta$-bag-chain $B_1, \dots, B_8$ where $B_1, B_2$ is the $\beta$-bag-chain of $A_1^1$, where $B_3, B_4$ is the $\beta$-bag-chains of $A_1^2$, and so on. But then $B_1, \dots, B_8$ is a $\beta$-bag-chain of length 8, finishing the proof.
\end{subproof}

With the first objective achieved, we now prove the second objective. From now on, we assume $B_1, \dots, B_t$ is a $\beta'$-bag-chain in $D$ with $t$ maximum, where $\beta' = 2|V(H)|(c+b')+b'+1$. For convenience, define $B_{i, j}$, where $i\leq j$, as the union of the bags $B_i, \dots, B_j$.

\sta{\label{minor1}
$\dichi(N^0(v))\leq b'$ for every $v\in V(D)$.
}
\begin{subproof}
    This is a consequence of the fact that the set of non-neighbours of $v$ is $\{F, \Delta(1, 1, H) \}$-free. Thus, the result holds by the definition of $b'$.
\end{subproof}

The following is the equivalent of Claim 4.3 in \cite{dense-digraphs}, although we are able to prove a stronger statement.

\sta{\label{minor2}
For every $i\geq 1$, $v\in B_i$, and $s> 1$,
}
\begin{itemize}
    \item $N^+(v)\cap B_{i-s}=\emptyset$, and

    \item $N^-(v)\cap B_{i+s}=\emptyset$.
\end{itemize}

\begin{subproof}
    For a contradiction, let $s>1$ be the smallest integer such that there exist vertices $u$ and $v$ such that $u\in N^+(v)\cap B_{i-s}$ or $u\in N^-(v)\cap B_{i+s}$. We deal first with the former. 

    Suppose first that $s = 2$. Let $A = (N^-(u)\cup N^+(v))\cap B_{i-1}$, and $B = (N^0(u)\cup N^0(v))\cap B_{i-1}$. By the definition of a $\beta$-bag-chain and (\ref{minor1}), $\dichi(A)\leq 2c$ and $\dichi(B)\leq 2b'$. Thus, $\dichi(B_{i-1}\setminus (A\cup B))\geq \beta'-2c-2b'>c$. By the definition of $c$, there exists a copy $X$ of $H$ in $B_{i-1}\setminus (A\cup B)$. But by the definition of $A$ and $B$, this implies that $\{u, v \}\cup X$ induces a copy of $\Delta(1, 1, H)$, a contradiction.

    Suppose then that $s>2$. The proof for this case is very similar. Let $A = (N^-(u)\cup N^+(v))\cap B_{i-1})$ and $B=(N^0(u)\cup N^0(v))\cap B_{i-1}$. By the minimality of $s$, and since $s>1$, we have $A = N^+(v)\cap B_{i-1}$. By (\ref{minor1}),  it follows that $\dichi(B)\leq 2b'$. Thus, $\dichi(B_{i-1}\setminus (A\cup B))\geq \beta'-2b'-c>c$. By the definition of $c$, there exists a copy $X$ of $H$ in $B_{i-1}\setminus (A\cup B)$. But then, by the definition of $A$ and $B$, this implies that $\{u, v\}\cup X$ induces a copy of $\Delta(1, 1, H)$, a contradiction. 

    The proof for the case where $u\in N^-(v)\cap B_{i+s}$ is analogous with arcs reversed.
\end{subproof}

The following is the equivalent of Claim 4.4 and Claim 4.5 in \cite{dense-digraphs}.

\sta{\label{minor3}
For every $i$ and $v\in B_i$,
}
\begin{itemize}
    \item $\Vec{\chi}(N^+(v)\cap B_{1, i-1})\leq c$, \textit{and}

    \item $\Vec{\chi}(N^-(v)\cap B_{i+1, t})\leq c$.
\end{itemize}
\begin{subproof}
    The result is immediate from (\ref{minor2}) and the definition of $\beta'$-bag-chains.
\end{subproof}

We can now prove our second objective:

\sta{\label{major1}
$\dichi(B_{1, t})\leq 6(\max\{b', c \}+\beta')+2$.
}
\begin{subproof}
    Applying Lemma \ref{partition-dichi} with $m = \max\{b', c \}$, and $m'=\beta'$, where the hypothesis holds by (\ref{minor3}), it follows that $\Vec{\chi}(B_{1, t})\leq 6(\max\{b', c \}+\beta')+2$.
\end{subproof}

For our final objective, we will partition the  vertices of $V(D)\setminus B_{1,t}$ in such a way that they behave similarly to a bag chain as well. We partition $V(D)\setminus B_{1,t}$ into sets $Z_i$ we call \textit{zones} such that $v\in Z_i$ if $i$ is the largest index such that $\dichi(N^-(v)\cap B_i)>c$, and $v\in Z_0$ if no such $i$ exists. Furthermore, for convenience, set $Z_{i, j} := Z_i\cup \dots \cup Z_j$ for $i\leq j$. We proceed to prove claims that will allow us to bound $\dichi(Z_{0, t})$ by using Lemma \ref{partition-dichi}. To this end, in (\ref{zones1}--\ref{zones3}), we will show that zones interact with the bag chain and each other in limited ways. 

\sta{\label{zones1}
For every $i$ and every $v\in Z_i$,
}
\begin{itemize}
    \item $\dichi(N^-(v)\cap B_{i+r})\leq c$ for $r\geq 1$, and
    \item $N^+(v)\cap B_{i-r}=\emptyset$ for $r\geq 2$.
\end{itemize}
\begin{subproof}
    The first bullet point is true by the definition of $Z_i$. We prove the second. For a contradiction, assume that there exists a vertex $u$ such that $u\in N^+(v)\cap B_{i-r}$. We claim that $\dichi(N^-(v)\cap B_{i-1})\leq b'+2c$. For a contradiction, assume this is not the case. Set
    $$
    A:=(N^0(u)\cup N^-(u))\cap (N^-(v)\cap B_{i-1}).
    $$
    By (\ref{minor1})  and (\ref{minor3}), $\dichi(A)\leq b'+c$, so $\dichi((N^-(v)\cap B_{i-1})\setminus A)>c$. Thus, there exists a copy $X$ of $H$ in $(N^-(v)\cap B_{i-1})\setminus A$. But then $\{u, v \}\cup X$ induces a copy of $\Delta(1, 1, H)$, a contradiction.

    Thus, $\dichi(N^-(v)\cap B_{i-1})\leq b'+2c$. Since $\dichi(N^0(v)\cap B_{i-1})\leq b'$ by (\ref{minor1}), and since $B_{i-1}$ is a $\beta'$-bag, it follows that $\dichi(N^+(v)\cap B_{i-1})\geq |V(H)|(b'+c)+1$. By the definition of a zone, there exists a copy $X'$ of $H$ in $N^-(v)\cap B_i$. Set
    $$
    A':=\bigcup_{x\in X'}(N^0(x)\cup N^+(x))\cap (N^+(v)\cap B_{i-1}).
    $$
    By (\ref{minor1}) and (\ref{minor3}), it follows that $\dichi(A')\leq |V(H)|(b'+c)$. Thus, $\dichi((N^+(v)\cap B_{i-1})\setminus A')>0$, so there exists a vertex $u'$ in $(N^+(v)\cap B_{i-1})\setminus A'$. This, however, implies that $\{u', v, X' \}$ induces a copy of $\Delta(1, 1, H)$, a contradiction.
\end{subproof}

\sta{\label{zones2}
For every $i\geq 0$, $v\in B_i$, and $r\geq 2$, we have $N^+(v)\cap Z_{i-r}=\emptyset$.
}
\begin{subproof}
    For a contradiction, assume that there exists a vertex $u$ such that $u\in N^+(v)\cap Z_{i-r}$. Now let
    $$
    A = (N^0(u)\cup N^-(u))\cap B_{i-1},
    $$
    and let
    $$
    B = (N^+(v)\cup N^0(v))\cap B_{i-1}.
    $$
    By the definition of zones and by (\ref{minor1}), $\dichi(A)\leq b'+c$, and by the definition of a $\beta$-bag-chain and (\ref{minor1}), $\dichi(B)\leq b'+c$. Thus, $\dichi(B_{i-1}\setminus (A\cup B))\geq \beta'-(b'+c) - (b'+c)>c$. Consequently, there exists a copy $X$ of $H$ in $B_{i-1}\setminus (A\cup B)$. But by the definition of $A$ and $B$, $\{u, v \}\cup X$ induces a copy of $\Delta(1, 1, H)$, a contradiction.
\end{subproof}
    
\sta{\label{zones3}
For every $i$, $v\in B_i$, and $r\geq 3$, we have $N^-(v)\cap Z_{i+r}=\emptyset$.
}
\begin{subproof}
    For a contradiction, assume that there exists a vertex $u$ such that $u\in N^-(v)\cap Z_{i+r}$. Now let
    $$
    A := (N^0(u)\cup N^+(u))\cap B_{i+1},
    $$
    and let
    $$
    B := (N^0(v)\cup N^-(v))\cap B_{i+1}.
    $$
    By (\ref{minor1}) and (\ref{zones2}), $\dichi(A)\leq b'$. Furthermore, by (\ref{minor1}) and the definition of bags, $\dichi(B)\leq b'+c$. Thus, $\dichi(B_{i+1}\setminus (A\cup B))  \geq \beta'-b'-(b'+c)>c$. Consequently, there exists a copy $X$ of $H$ in $B_{i+1}\setminus (A\cup B)$. But by the definition of $A$ and $B$, it follows that $\{u, v\}\cup X$ induces a copy of $\Delta(1, 1, H)$, a contradiction.
\end{subproof}

Finally, we are ready to bound $\dichi(Z_i)$. The following is the equivalent of Claim 4.10 in \cite{dense-digraphs}.

\sta{\label{zones4}
For every $i$, $\dichi(Z_i)\leq \hat{f}(\beta')$.
}
\begin{subproof}
    By (\ref{balls2}), it is enough to prove that zones do not have a $\beta'$-bag-chain of length 8. We will do this by using the maximality of $t$. Assume for a contradiction that $Y_1, \dots, Y_8$ is a $\beta'$-bag-chain of length 8 in $Z_i$. By (\ref{zones1}), (\ref{zones2}) and (\ref{zones3}), $B_1, \dots, B_{i-3}, Y_1, \dots, Y_8, B_{i+3}, \dots, B_t$ is a longer $\beta'$-bag-chain than $B_1, \dots, B_t$ which contradicts the maximality of $t$.
\end{subproof}

To finish the proof, it remains to show we can partition $Z_{0, t}$ such that we are able to colour each part. The following is the equivalent of Claim 4.9 in \cite{dense-digraphs}.

\sta{\label{zones5}
For every $i$ and $v\in Z_i$,
}
\begin{itemize}
    \item $N^+(v)\cap Z_{0, i-3}=\emptyset$, and
    \item $N^-(v)\cap Z_{i+3, t}=\emptyset$.
\end{itemize}
\begin{subproof}
    Let us prove the first bullet point. Suppose for a contradiction that there exists a vertex $u$ such that $u\in N^+(v)\cap Z_{0, i-3}$. Now let
    $$
    A := (N^0(u)\cup N^-(u))\cap B_{i-2},
    $$
    and
    $$
    B:=(N^0(v)\cup N^+(v))\cap B_{i-2}.
    $$
    By (\ref{minor1}) and the definition of zones, $\dichi(A)\leq b'+c$. Similarly, $\dichi(B)\leq b'$ by (\ref{minor1}) and (\ref{zones1}). Since $B_{i-2}$ is a $\beta'$-bag, we have $\dichi(B_{i-2}\setminus (A\cup B))>\beta'-(b'+c)-b'>c.$ By the definition of $c$, there exists a copy $X$ of $H$ in $B_{i-2}\setminus(A\cup B)$. But then, by the definitions of $A$ and $B$, it follows that $\{u, v \}\cup X$ induces a copy of $\Delta(1, 1, H)$, a contradiction.  A similar argument, using the established claims, gives the second bullet point.
\end{subproof}

We are ready to prove that $\dichi(Z_{0, t})$ is bounded.

\sta{\label{zones6}
$\dichi(Z_{0, t})\leq 3\hat{f}(\beta')$.
}
\begin{subproof}
    Let $\mathcal{Z}_i=\bigcup_{j\cong i\mod{3}}Z_j$. By (\ref{zones5}), every strongly connected component in $\mathcal{Z}_i$ is contained in a zone $Z_j$. Thus, by (\ref{zones4}), $\dichi(\mathcal{Z}_i)\leq \hat{f}(\beta')$. Since $\mathcal{Z}_1, \mathcal{Z}_2, \mathcal{Z}_3$ is a partition of $Z_{0, t}$, it follows that $\dichi(Z_{0, t})\leq 3\hat{f}(\beta')$ as claimed. 
\end{subproof}

We are ready to finish the proof. Since $V(D) = B_{1, t}\cup Z_{0, t}$, and by (\ref{major1}) and (\ref{zones6}), we have:
$$
\dichi(D)\leq \dichi(B_{1, t})+\dichi(Z_{0, t}) \leq 6(\max\{b', c \}+\beta') +2+3\hat{f}('\beta)
$$
as claimed.
\qed{}


\section{Forbidding $rK_1 + \vec{P_3}$} \label{sec:heroes}

In this section, we prove Theorem \ref{main2}, which we restate for the reader's convenience.

\maintwo*

Equivalently, we will prove that for every $r\geq 1$, the digraph $rK_1+\vec{P_{3}}$ cooperates. We will use Theorem \ref{main1} to do this. Thus, we need to prove that $\vec{P_{3}}$ cooperates, and that $\vec{P_{3}}$ is localized and colocalized. The fact that $\vec{P_3}$ cooperates is a consequence of Theorem \ref{thm:p3}. Notice that by Lemma \ref{fix}, we only need to show that $\vec{P_3}$ is localized. 

To prove that $\vec{P_3}$ is localized, we use domination. We say a set of vertices $S_1$ \textit{dominates} a set of vertices $S_2$, or equivalently $S_1$ is a \textit{dominating set} for $S_2$, if every vertex in $S_2\setminus S_1$ is seen by a vertex in $S_1$. A digraph $F$ \textit{dominates} if, for every $r\geq 1$, the following two:
\begin{itemize}
    \item $\Delta(1, 1, H)$ is a hero in $\{(r-1)K_1+F\}$-free digraphs;
    \item $H$ is a hero in $\{rK_1+F\}$-free digraphs;
\end{itemize}
imply that there exists a function $g(r, k, H)$ such that for every $\{\Delta(1, 1, H), rK_1+F \}$-free $k$-local digraph $D$, either $\Vec{\chi}(D)\leq g(r, k, H)$, or $F$-free acyclic induced subsets $S$ of $V(D)$ have a dominating set in $D$ of size at most $g(r, k, H)$. While this definition is rather technical, it allows us to formulate a proof in such a way that parts of it are more general than the case of $\vec{P_3}$. 

We want to prove that if $F$ dominates, then $F$ is localized. The concept that a digraph $F$ dominates, as well as how this implies that $F$ is localized, is meant to generalize the proof strategy  devised by Harutyunyan, Le, Newman, and Thomassé \cite{dense-digraphs} to prove that $k$-local $rK_1$-free digraphs, where $r\geq 2$, have bounded dichromatic number.

To prove that digraphs that dominate are localized, we use a concept introduced in \cite{dense-digraphs}. A family of digraphs $\mathcal{C}$ is \textit{tamed} if, for every $m$, there exists integers $M$ and $l$ such that if $D\in \mathcal{C}$ has $\dichi(D)\geq M$, then there exists a subset $X\subseteq V(D)$ such that $|X|\leq l$ and $\dichi(X)\geq m$. The following proof is a slight generalization of the proof of Claim 2.4 in \cite{dense-digraphs}.

\begin{lemma}\label{tame2}
If $F$ dominates and the following two hold:
\begin{itemize}
    \item $\Delta(1, 1, H)$ is a hero in $\{(r-1)K_1+F\}$-free digraphs, and
    \item $H$ is a hero in $\{rK_1+F\}$-free digraphs,
\end{itemize}
then, for every $k\geq 1$, the family of $\{\Delta(1, 1, H), rK_1+F \}$-free $k$-local digraphs is tamed.
\end{lemma}
\begin{proof}
    We proceed by induction on $m$ (from the definition of tamed). The case when $m=1$ is immediate. Assume the statement holds for $m$. Let $M$ and $l$ be the corresponding integers. Let $c$ be an integer such that $\{rK_1+F, H \}$-free digraphs $D$ have $\dichi(D)\leq c$, and let $b$ be an integer such that $\{(r-1)K_1+F, \Delta(1, 1, H) \}$-free digraphs $D$ have $\dichi(D)\leq b$. Since $F$ dominates, let $g(r, k, H)$ be the associated function. Furthermore, let $p=M+bl+kl+1$, and let $d=m((g(r, k, H)+r)p+1)+1$. Note that, by the pigeonhole principle, $d$ is the smallest number such that if a set $S$ of size $d$ is $m$-coloured, then there exists a monochromatic susbset of size at least $(g(r, k, H)+r)p+2$. We claim that the statement holds for $m+1$ when $M' = \max\{g(r, k, H)+1, kd, M+d(b+k+1)\}$ and $l' = d + l + l{d \choose (g(r, k, H)+r)p+2}$.

    Assume that $D$ is a $\{\Delta(1, 1, H), rK_1+F \}$-free $k$-local digraph, and assume $\Vec{\chi}(D)\geq M'$. We start with the following claim.

    \sta{\label{tame1}
    $\dichi(N^0(v))\leq b$ for every $v\in V(D)$.
    }
\begin{subproof}
    Since $D$ is $\{rK_1+F, \Delta(1, 1, H)\}$-free, it follows that $N^0(v)$ is $\{(r-1)K_1+F, \Delta(1, 1,$ $H)\}$-free, so the claim follows by definition of $b$. 
\end{subproof}
    
    Since $\dichi(D)\geq M'$, we have $\Vec{\chi}(D)>g(r, k, H).$ Let $B$ be a minimum dominating set for $D$. Since $D$ is $k$-local, it follows that $\Vec{\chi}(D)\leq |B|k$, so $|B|\geq M'/k\geq d$. Pick $W\subseteq B$ such that $|W|= d$. By the choice of $M'$ and the size of $B$, we know this subset exists. Notice that $\dichi(\bigcup_{w\in W}N^0(w))\leq bd$ by (\ref{tame1}), and $\Vec{\chi}(\bigcup_{w\in W}N^+(w))\leq kd$ since $D$ is $k$-local. Since $\dichi(D\setminus W)\geq M'-d$, it follows that the set $\mathcal{A}$ of vertices out-complete to $W$ has dichromatic number at least $M'-d-bd-kd\geq M.$ By the definition of $M$, there exists a set $A$ out-complete to $W$ of size at most $l$ and dichromatic number at least $m$.

    We will define a set $A_S$ for every subset $S$ of $W$ of size $(g(r, k, H)+r)p+2$ as follows. Let $S$ be such a set, and let $Y = \bigcup_{s\in S}N^+(s)$. For a contradiction, assume that $\dichi(Y)\leq p$. Let $Y_1, \dots, Y_p$ be a partition of $Y$ into $p$ acyclic sets. For each set $Y_i$, pick a vertex $y_i^1$ with no in-neighbours. Having picked vertex $y_i^j$ for some $1\leq j\leq r-1$, pick another vertex $y_i^{j+1}$ in $Y_i\setminus \bigcup_{k\leq j}N^+[y_i^k]$ (unless this set is empty)
    with no in-neighbours in $Y_i\setminus \bigcup_{1\leq k\leq j}N^+[y_i^k]$. Then, for every $i$, the vertices $y_i^1, \dots, y_i^{r}$ form a stable set, and so the set $Y_i' = Y \setminus \bigcup_{1\leq k\leq r}N^+[y_i^k]$ is acyclic and $F$-free. Since $\dichi(D) > g(r, k, H)$, there exists a dominating set $Z_i$ for $Y_i'$ of size at most $g(r, k, H)$, so the set $Z_i' = Z_i\cup \{y_i^1, \dots, y_i^r \}$ is a dominating set for $Y_i$ of size at most $g(r, k, H)+r$. 
    
    Thus, the set $Z = Z_1'\cup \dots \cup Z_p'$ is a dominating set for $Y$ of size at most $(g(r, k , H)+r)p$. Adding a vertex $z$ from $A$, we get a dominating set for $N^+[S]$ of size at most $(g(r, k, H)+r)p+1$. Then $(B\setminus S)\cup Z\cup\{z \}$ is a dominating set for $D$ of size at most $|B|-1$, contradicting that $B$ is a smallest dominating set. Thus, $\dichi(Y) > p$.

    Because $|A|\leq l$, by (\ref{tame1}), and by the fact that $D$ is $k$-local, we have
    $$
    \Vec{\chi}(N^0(A)\cap Y)\leq bl,
    $$
    and 
    $$
    \Vec{\chi}(N^+(A)\cap Y)\leq kl.
    $$
    Thus, the set $A'$ of vertices of $Y$ out-complete to $A$ has dichromatic number at least $p-bl-kl>M$, which implies by the inductive hypothesis that $A'$ contains a set $A_S$ with $\Vec{\chi}(A_S)\geq m$ and $|A_S|\leq l$. This is how we define $A_S$ for every subset $S$ of $W$ where $|S|= (g(r, k, H)+r)p+2$. Figure \ref{figures-heroes-1} illustrates this process.

\begin{figure}
    \centering
    \includegraphics[scale=0.9]{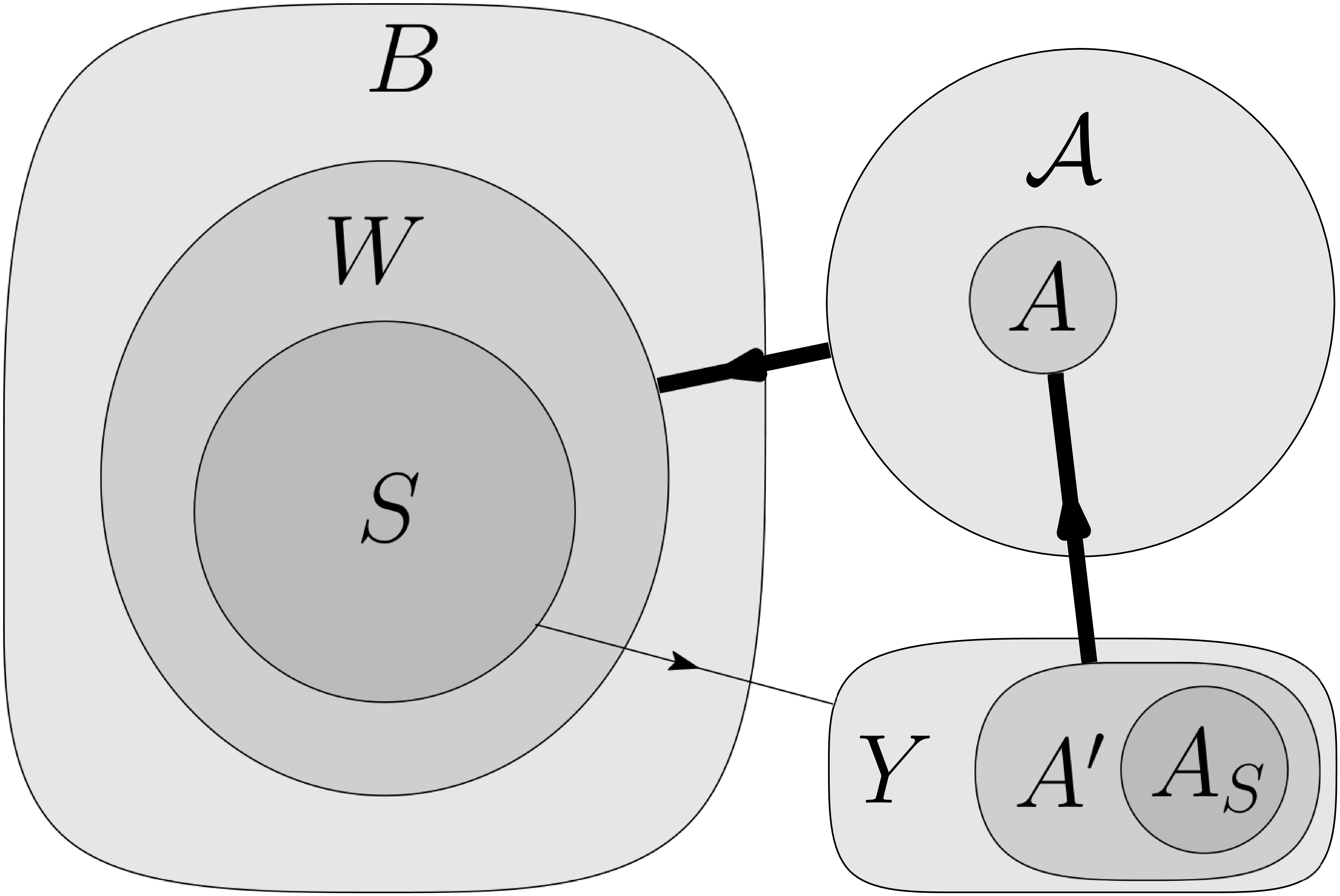}
    \caption{Illustration of the proof of Lemma \ref{tame2}.}
    \label{figures-heroes-1}
\end{figure}

    Finally, take
    $$
    V := W\cup A\cup \bigcup A_S.
    $$
    where the union happens over all subsets $S$ of $W$ of size exactly $(g(r, k, H)+r)p+2$. This set has size at most $d + l + l{d \choose g(r, k, H)p+2} = l'.$  By the definition of $d$, every $m$-colouring $f$ of $V$ contains a monochromatic set $S\subseteq W$ of size $g(r, k, H)p+2$. Let $f(S) = \{\gamma \}$. Since $\dichi(A), \dichi(A_S)\geq m$, it follows that there exists $a\in A$ and $a'\in A_S$ with $f(a)=f(a')=\gamma$. Now let $s\in S$ be an in-neighbour of $a'$ (which exists since $A_S\subseteq N^+(S)$). It follows that $\{a, a', s \}$ is a cyclic triangle monochromatic under $f$. Since $f$ was an arbitrary $m$-colouring, this argument applies to every $m$-colouring of $V$. We conclude that $\Vec{\chi}(V)\geq m+1$, and so $V$ is the desired set for $m+1$, finishing the inductive argument. \end{proof}

The following is analogous to the proof of Theorem 2.3 in \cite{dense-digraphs}.

\begin{lemma}\label{k_2_1}
    If $F$ dominates, then $F$ is localized.
\end{lemma}
\begin{proof}
    Assume that 
    \begin{itemize}
    \item $\Delta(1, 1, H)$ is a hero in $\{(r-1)K_1+F\}$-free digraphs, and
    \item $H$ is a hero in $\{rK_1+F\}$-free digraphs.
\end{itemize}
    Let $c$ be an integer such that $\{rK_1+F, H \}$-free digraphs $D$ have $\dichi(D)\leq c$. Furthermore, let $b$ be an integer such that $\{(r-1)K_1+F, \Delta(1, 1, H) \}$-free digraphs $D$ have $\dichi(D)\leq b$. Fix an integer $k\geq 1$. By Lemma \ref{tame2}, $\{rK_1+F, \Delta(1, 1, H) \}$-free $k$-local digraphs are tamed. Let $M$ and $l$ be the corresponding integers following the definition of tameness when $m = k+b+1$. 
    
    Let $D$ be a $\{ rK_1+F, \Delta(1, 1, H)\}$-free $k$-local digraph. To prove that $F$ is localized, it is enough to show that $\dichi(D)\leq \max\{M, lk\}$. Assume that $\dichi(D)>M$. By definition, there exists a set $X\subseteq D$ such that $|X|\leq l$ and $\dichi(X)\geq m$. We claim that $X$ is a dominating set of $D$. Assume for a contradiction that there exists a vertex $v$ not in $\bigcup_{x\in X}N^+(x)$. Consequently, $X\subseteq N^0(v)\cup N^+(v)$. By the definition of $k$ and $b$, it follows that $\dichi(X)\leq k+b$, a contradiction. Thus, $X   $ dominates $D$. But $D$ is $k$-local, so $\dichi(D)\leq kl$ thus finishing the proof. 
    \end{proof}

Now that we have proven that digraphs $F$ that dominate are localized, it only remains to show that $\vec{P_{3}}$ dominates. 

\begin{lemma}
\label{k_2_2}
The digraph $\vec{P_{3}}$ dominates. 
\end{lemma}

\begin{proof}
Suppose that $D$ is a $k$-local $\{\Delta(1,1,H), rK_{1} + \vec{P_{3}}\}$-free digraph. Suppose that 
\begin{itemize}
\item $\Delta(1,1,H)$ is a hero in $\{(r-1)K_{1} + \vec{P_{3}}\}$-free digraphs; and
\item $H$ is a hero in $\{rK_{1} + \vec{P_{3}}\}$-free digraphs.
\end{itemize}
We set some constants:
\begin{itemize}
\item Let $b$ be an integer such that $\{\Delta(1,1,H), (r-1)K_{1} + \vec{P_{3}}\}$-free digraphs have dichromatic number at most $b$.
\item Let $c$ be an integer such that $\{H, rK_{1} + \vec{P_{3}}\}$-free digraphs have dichromatic number at most $c$. 
\end{itemize}

Let $g(r,k,H) = \max\{4r+5,  b + 1 + k + 2c + (|V(H)|+1)(kr + b), 3r|V(H)|\}$. We will show that either $D$ has dichromatic number at most $g(r,k,H)$, or for every acyclic $\vec{P_{3}}$-free set $S$, there is a dominating set in $D$ for $S$ of size at most $g(r,k,H)$. Suppose for a contradiction that neither of these outcomes holds.

Let $S$ be an acyclic $\vec{P_{3}}$-free set. By possibly adding vertices to $S$, we assume that $S$ is a vertex-maximal acyclic $\vec{P_{3}}$-free set. As $S$ is maximal, all vertices in $V(D) \setminus S$ have a neighbour in $S$. We start by noting that acyclic $\vec{P_{3}}$-free digraphs with small independence number can be dominated with few vertices. 

\sta{
\label{smallindependencenumber}
Suppose that $X$ is an acyclic $\vec{P_{3}}$-free digraph with independence number $q$. Then there is a dominating set $B$ of $X$ contained inside $X$ of at most $q$ vertices. 
}

\begin{subproof}
Let $B\subseteq V(X)$ be a minimal dominating set for $X$, and suppose that $B$ contains at least $q+1$ vertices. As the independence number of $X$ is at most $q$, there is an arc $uv$ in $X[B]$. As $B \setminus \{v\}$ is not a dominating set, there is a vertex $w$ such that $vw \in A(X)$ but $uw \not \in A(X)$. If $wu \not \in A(X)$, then $X$ contains an induced $\vec{P_{3}}$, a contradiction. So $wu \in A(X)$. But then $\{u,v,w\}$ induces a cyclic triangle, contradicting that $X$ is acyclic. 
\end{subproof}

\sta{\label{inneighbourordominating}
If $Y$ is an induced copy of $\vec{P_3}$ in $D \setminus S$, then a vertex in $Y$ has an in-neighbour in $S$.
}

\begin{subproof}
Suppose not. Then, no vertex in $Y$ has an in-neighbour in $S$. Let $S' = S \cap N^{+}(Y)$. As no vertex in $Y$ has an in-neighbour in $S$, the set of vertices $v \in S \setminus S'$ are common non-neighbours of all of the vertices in $Y$. Thus $S \setminus S'$ has independence number at most $r-1$, as otherwise $D$ contains $rK_{1} + \vec{P_{3}}$. Thus there is a dominating set $B$ for $S \setminus S'$ of size at most $r-1$ by \eqref{smallindependencenumber}. But then $Y \cup B$ is a dominating set for $S$ of size at most $r+2$, a contradiction. 
\end{subproof}

\sta{
\label{inneighbourorsmall}
If $X$ is an induced subgraph of $D \setminus S$ and $N^{-}(X) \cap S = \emptyset$, then $\vec{\chi}(X) \leq b$.
}

\begin{subproof}
By \eqref{inneighbourordominating}, we conclude that $X$ is $\vec{P_{3}}$-free. Then, by the definition of $b$, it follows that $\vec{\chi}(X) \leq b$. 
\end{subproof}

Let $S_{1} \subseteq S$ be the set of vertices in $S$ with no in-neighbour in $S$.  

\sta{
\label{s1defn}
The set $S_{1}$ is a stable set, and all vertices in $S-S_{1}$ have an in-neighbour in $S_{1}$.
}
\begin{subproof}
    This is immediate from the fact that $\vec{P_3}$-free acyclic digraphs are directed comparability graphs. We give a self-contained proof for completeness.
    
    The fact that $S_1$ is a stable set follows directly from the definition of $S_1$. Now let $v \in S \setminus S_1$. Let $P$ be a maximal directed (not necessarily induced) path of the form $x_1 \rightarrow x_2 \rightarrow \dots \rightarrow x_t \rightarrow v$ in $D[S]$. We claim that $x_1 \in S_1$. If not, then $x_1$ has an in-neighbour $x_0 \in S$. Since $x_0$ cannot be added to $P$ to make a longer path, it follows that $x_0 \in \{x_2, \dots, x_t, v\}$. But then $D[S]$ has a directed cycle, a contradiction. So $x_1 \in S_1$. Now let $Q$ be a shortest directed path from $x_1$ to $v$. Then $Q$ is an induced path (since $D[S]$ is acyclic); but since $S$ is $\vec{P_3}$-free, it follows that $Q$ has at most one edge; in other words, $x_1v \in A(D)$. Since $v$ was chosen arbitrarily, the claim follows. 
\end{subproof} 

We observe that $S_{1}$ is a dominating set for $S$, and thus $|S_{1}| \geq g(r,k,H) + 1$ by our assumptions. 

\sta{
\label{smalldichidominating}
Let $Q \subseteq V(D) \setminus S$ be the set of vertices such that for each vertex $v \in Q$, we have that the in-neighbours of $v$ in $S$ can be dominated by a set $B$ of at most $r+1$ vertices where $B 
\subseteq S$. Then $\vec{\chi}(Q) \leq b$. 
}

\begin{subproof}

Suppose not. Then $Q$ contains an induced copy $P$ of $\vec{P_{3}}$. Partition $S$ into $(N^{-}(P) \cap S),(N^{0}(P) \cap S),$ and $(N^{+}(P) \cap S)$ (choosing arbitrarily if a vertex is both an in- and out-neighbour of some vertex in $P$). Note this is a partition of $S$: if a vertex is neither an in-neighbour or out-neighbour of a vertex in $P$, then it is a non-neighbour of all of the vertices of $P$. Then the digraph induced by $N^{0}(P) \cap S$ has independence number at most $r-1$, as otherwise $D$ contains $rK_{1} + \vec{P_{3}}$, a contradiction. Thus by \eqref{smallindependencenumber}, there exists a dominating set for $N^{0}(P) \cap S$ of size at most $r-1$. By the assumption, $N^{-}(P) \cap S$ can be dominated by at most $3(r+1)$ vertices. Lastly, $N^{+}(S)$ is dominated by $P$, and thus $S$ can be dominated by at most $3(r+1) + r + 2 = 4r + 5 \leq g(r,k,H)$ vertices, a contradiction. 
\end{subproof}

Let $Q$ be the set of vertices defined as in \eqref{smalldichidominating}. Let $T = D \setminus (S \cup Q)$. As $\vec{\chi}(Q) \leq b$, and $S$ is acyclic, it follows that $\vec{\chi}(T) \geq \vec{\chi}(D) - b- 1$. 

\sta{
\label{fewnonneighboursfirstlayer}
If there exists a vertex $v \in V(T)$ such that $v$ has an out-neighbour in $S_{1}$, then $v$ has at most $r-1$ non-neighbours in $S_{1}$.}

\begin{subproof}
Let $v$ be a vertex in $T$ and suppose that $v$ has at least $r$ non-neighbours in $S_{1}$. Let $u$ be an out-neighbour of $v$ in $S_{1}$. Let $X$ be any set of $r$ non-neighbours of $v$ in $S_{1}$. The set $S \cap N^{-}(v)$ cannot be dominated by $X \cup \{u\}$, as the in-neighbours of $v$ in $S$ cannot be dominated by $r+1$ vertices by the definition of $Q$, and thus there is at least one in-neighbour of $v$ in $S$, say $w$, such that $w$ is not in $N^+(X \cup \{u\})$. Since $N^-(S_1) \cap S = \emptyset$, it follows that $w$ is not adjacent to any vertex in $X \cup \{u\}$. Then $\{w,v,u\} \cup X$ induces an $rK_{1} + \vec{P_{3}}$, a contradiction.  
\end{subproof}

\sta{
\label{nonneighbourfirstlayer}
If $v$ is in $T$, and $v$ has an out-neighbour in $S_{1}$, then the non-neighbours of $v$ can be dominated with at most $\max\{r+1, 2r-1\} \leq 2r$ vertices inside $S$. 
}

\begin{subproof}
Let $u$ be an out-neighbour of $v$ in $S_{1}$. First suppose that $v$ has an in-neighbour in $S_{1}$, say $w$. Then $\{u,v,w\}$ induces a $\vec{P_{3}}$. Let $Y = S \cap N^{0}(\{u, v, w\})$. Then, since $\{u, v, w\}$ induces a copy of $\vec{P_3}$, we have that $Y$ has independence number at most $r$, and thus it follows from \eqref{smallindependencenumber} that $Y$ has a dominating set $X$ of size at most $r-1$. Consequently, the non-neighbours of $v$ in $S$ can be dominated by $X  \cup \{u,w\}$, which is at most $ r-1 +2 = r+1$ vertices. 

Therefore, we may assume that $v$ only has out-neighbours and non-neighbours in $S_{1}$. By \eqref{fewnonneighboursfirstlayer}, $v$ has at most $r-1$ non-neighbours in $S_{1}$. Let $X$ be this set. Let $Y \subseteq S_1 \setminus X$ be minimal with respect to inclusion such that $(S \setminus S_1) \cap N^0(v) \subseteq N^+(X \cup Y)$. This set exists as $S \setminus S_1 \subseteq N^+(S_1)$ by \eqref{s1defn}. Then, if $|Y| \leq r$, the claim holds as $X \cup Y$ is the desired set; so we may assume that $|Y| \geq r+1$. It follows that $Y$ contains $r+1$ distinct vertices, say $y_1, \dots, y_{r+1}$. For each $i \in \{1, \dots, r+1\}$, the set $X \cup (Y \setminus \{y_i\})$ does not dominate $(S \setminus S_1) \cap N^0(v)$, and so there is a vertex $y_i' \in (S \setminus S_1) \cap N^0(v)$ such that $N^-(y_i') \cap (X \cup Y) = \{y_i\}$. But now $\{v, y_1, y_1', y_2', \dots, y_{r+1}'\}$ induces $rK_1 + \vec{P_3}$, a contradiction.  
\end{subproof}

From now on, let $X$ be the set of vertices in $V(D) \setminus S$ with no out-neighbour in $S_{1}$.

\sta{\label{allinneighbours}
Either $|N^{0}(X) \cap S_{1}| \leq r-1$, or $\vec{\chi}(X) \leq b$.}

\begin{subproof}
If not, then $|N^{0}(X) \cap S_{1}| \geq r$, and thus $X$ does not induce a $\vec{P_{3}}$, as otherwise $D$ would contain a copy of $rK_{1} + \vec{P_{3}}$. But then by the definition of $b$, it follows that $\vec{\chi}(X) \leq b$, a contradiction. 
\end{subproof}

Since $D$ is $k$-local, \eqref{allinneighbours} implies that $\dichi(X) \leq \max\{b, rk\}$ (because if $|N^{0}(X) \cap S_1| \leq r-1$, then choosing $r$ vertices in $S_1$ yields a dominating set for $X$). In addition, every vertex $v$ in $T \setminus X$ has an out-neighbour in $S_1$, and thus, by \eqref{nonneighbourfirstlayer}, we have $|S_1 \cap N^0(v)| \leq r-1$. 

\sta{\label{inneighboursmalldichi}Let $R$ be the set of vertices in $T \setminus X$ that have an in-neighbour in $S_{1}$. Then $\dichi(R) \leq (r-1)|V(H)|k  + (k + c) + (b+k)|V(H)|$.}

\begin{subproof}
Suppose not. By removing one vertex at a time from $S_{1}$, we create a subset $S'$ of $S_{1}$ such that 
$$\dichi(R) - (k+c) \leq \dichi(N^+(S') \cap R) < \dichi(R) - c$$
(which is possible as $D$ is $k$-local and $\dichi(R) > c$). Let $Z =  R \setminus N^+(S')$. Then $\dichi(Z) > c$, and it follows that there exists a copy $X'$ of $H$ in $Z$. 

Let $S''$ be the set of vertices $s \in S'$ such that $s$ is a neighbour of every vertex in $X'$, and note that from the definition of $Z$, we have that $s$ is an out-neighbour of every vertex in $X'$ in this case. It follows that $X'$ is out-complete to $S''$. As every vertex in $X' \subseteq T \setminus X$ has at most $r-1$ non-neighbours in $S'$, we have that $|S''| \geq |S'| -(r-1)|V(H)|$ (and thus implying $S''$ is non-empty). Let $Y = N^+(S'') \cap R$.  Then, as $D$ is $k$-local and from the choice of $S'$, it follows that 
\begin{align*}
    \vec{\chi}(Y) &\geq \vec{\chi}(N^+(S') \cap R) - (r-1)|V(H)|k \\ &\geq \vec{\chi}(R) -(r-1)|V(H)|k  - (k + c) \\ &> b|V(H)| + k|V(H)|.
\end{align*} 
 Let $A = \bigcup_{x \in X'}N^{0}(x) \cap Y$ and $B = \bigcup_{x \in X'}N^{+}(x) \cap Y$. As $N^{0}(x)$ is $\{(r-1)K_{1} + \vec{P_{3}}\}$-free for every $x \in D$, we have that $\vec{\chi}(A) \leq b|V(H)|$; and $\vec{\chi}(B) \leq k|V(H)|$ as $D$ is $k$-local. 
Thus $\vec{\chi}(Y\setminus (A \cup B)) \geq \vec{\chi}(Y) - b|V(H)| - k|V(H)| \geq 1$ and therefore $Y' = Y \setminus (A\cup B)$ is not empty. Let $y \in Y'$ and $s \in S''$ be an in-neighbour of $y$ in $S''$, which exists by the definition of $Y$. By definition, $s$ is in-complete from $X'$, and $X'$ is in-complete from $y$. Thus the set $\{s,y\} \cup X'$ induces a $\Delta(1,1,H)$, a contradiction. 
\end{subproof}

Putting this all together, as $D$ has large dichromatic number, by \eqref{inneighboursmalldichi} and since $\dichi(T) \geq \dichi(D) - b - 1$, it follows that the set of vertices $U = T \setminus (X \cup R)$ (where $R$ is defined as \eqref{inneighboursmalldichi}) with only out-neighbours and non-neighbours in $S_{1}$  has dichromatic number at least 
$$\dichi(U) \geq \dichi(D) - b - 1 - \max\{b, rk\} - k - c - |V(H)|(kr + b) > c.$$ 
As $D[U]$ has dichromatic number more than $c$ and is $\{rK_{1} + \vec{P_{3}}\}$-free, it contains a copy $X'$ of $H$. Since $U \subseteq T \setminus X$, and from the definition of $X$, it follows that each vertex of $X'$ has an out-neighbour in $S_1$, and therefore, by \eqref{fewnonneighboursfirstlayer}, at most $r-1$ non-neighbours in $S_{1}$. Let $Y'$ be the set of vertices in $S_1$ with a non-neighbour in $X'$. Then $|Y'| \leq (r-1)|V(H)|$. Moreover, by \eqref{nonneighbourfirstlayer}, there is a set $X''$ of at most $2r|V(H)|$ vertices in $S$ such that $X''$ dominates the set of all vertices in $S$ with a non-neighbour in $X'$. 

If $Z' = X' \cup X'' \cup Y'$ is a dominating set for $S$, then it has size at most $3r|V(H)| \leq g(r,k,H)$, a contradiction. Therefore, there is a vertex $s$ such that:
\begin{itemize}
\item $s \in S$ is not an out-neighbour and not a non-neighbour of any vertex in $X'$, so $s$ is out-complete to $X'$; in particular, $s \not\in S_1$; and
    \item $s \in S \setminus N^+(Y')$, and so, since $S_1$ is a dominating set for $S$, it follows that $s$ has an in-neighbour $s'$ in $S_1 \setminus Y'$. As $s' \not\in Y'$, it follows that $X'$ is out-complete to $s'$.
\end{itemize}
But now $s$, $s'$ and $X’$ form a copy of $\Delta(1, 1, H)$, a contradiction.
\end{proof}

\noindent\textbf{Proof of Theorem \ref{main2}.} By Theorem \ref{main1}, it suffices to show that $\vec{P_{3}}$ cooperates, is localized, and is colocalized. It follows from Theorem \ref{thm:p3} that $\vec{P_{3}}$ cooperates. By Lemma \ref{k_2_2}, $\vec{P_{3}}$ dominates, so by Lemma \ref{k_2_1} $\vec{P_{3}}$ is localized. By Lemma \ref{fix}, $\vec{P_{3}}$ is colocalized as well, thus finishing the proof. 

\section{Forbidding brooms}\label{sec:brooms}

In this section, we prove Theorem \ref{dichi-main1}, which we restate for the reader's convenience.

\dichimainone*

As mentioned in the introduction, we follow the technique designed by Cook, Masařík, Pilipczuk, Reinald, and Souza \cite{p4-dichi-bounded} to prove that if $P$ is an orientation of $P_4$, then $P$-free digraphs are $\dichi$-bounded. We will need a lemma about so called $k$-nice sets. A set $S\not = \emptyset$ is $k$-nice if there exists a partition $S_1, S_2$ of $S$ such that every vertex in $S_1$ (resp. $S_2$) has at most $k$ in-neighbours (resp. $k$ out-neighbours) in $V(D)\setminus S$. Recall that a \textit{hereditary} class of digraphs $\mathcal{C}$ is a class of digraphs such that if $G \in \mathcal{C}$, all induced subdigraphs are in $\mathcal{C}$. 

\begin{restatable}{lemma}{knicesets}
\label{k-nice-sets}
    Let $k\geq 0$, and let $\mathcal{C}$ be a hereditary class of digraphs. If there exists an integer $c$ such that every $D\in \mathcal{C}$ has a $k$-nice set $S$ with $\dichi(S)\leq c$, then $\dichi(D)\leq 2c(k+1)$ for every $D\in \mathcal{C}$.
\end{restatable}

\begin{proof}
Fix $\mathcal{C}$. We proceed by induction on $|V(D)|$. The statement holds if $|V(D)|=1$. Assume the statement holds for digraphs with fewer than $|V(D)|$ vertices. By the assumption $D$ has a $k$-nice set $S$ with $\dichi(S)\leq c$. Let $S_1$ and $S_2$ be the partitioning of $S$ as in the definition of a $k$-nice set. 

By induction, the digraph induced by $V(D) \setminus S$ has a $2c(k+1)$-dicolouring. Let $f_0:(V(D)\setminus S)\rightarrow \{1, \dots, k+1 \}\times \{1, \dots, 2c\}$ be such a $(2c(k+1))$-dicolouring. Furthermore, let $f_1$ be a $c$-dicolouring of $S_1$ using colours in $\{1, \dots, c\}$, and let $f_2$ be a $c$-dicolouring of $S_2$ using colours in $\{c+1, \dots, 2c\}$. 

We define a function $m:S\rightarrow \{1, \dots, k+1\}$ as follows. Let $u\in S$. If $u\in S_1$, then $u$ has at most $k$ in-neighbours in $V(D)\setminus S$. Thus, $|f_0(N^-(u)\cap (V(D)\setminus S))|\leq k$. Consequently, there exists a number $m(u)$ such that no colour in $f_0(N^-(u)\cap (V(D)\setminus S))$ has $m(u)$ as its first coordinate. We define $m(v)$ when $v\in S_2$ similarly, where we use its out-neighbourhood in $V(D)\setminus S$ instead. Using these, we can define the following colouring.

\begin{align*}
f(v) = \left\{ \begin{array}{cc}
    f_0(v) & \hspace{4mm} \text{if }v\not\in S; \\
    (m(v), f_1(v)) & \hspace{5mm} \text{if }v\in S_1; \\
    (m(v), f_2(v)) & \hspace{5mm} \text{if }v\in S_2. \\
\end{array} \right.
\end{align*}

We claim that $f$ is a $(2c(k+1))$-dicolouring of $D$. The first index of the coordinate has $k+1$ values, and the second index at most $2c$. Thus, this indeed uses at most $2c(k+1)$ colours. For a contradiction, assume that $C$ is a directed monochromatic cycle in $D$. Since $f_0, f_1$ and $f_2$ are dicolourings, $C$ is not contained in neither of the sets $S_1$, $S_2$ and $V(D)\setminus S$. Since $f_1$ and $f_2$ use colours that do not overlap, it follows that $V(C)$ does not intersect both $S_1$ and $S_2$, so $V(C)$ is not contained completely in $S$. By the same reason, if $V(C)$ intersects both $S$ and $V(D)\setminus S$, then $V(C)$ intersects only one of $S_1$ and $S_2$. 

Thus, either $V(C)$ intersects with $V(D)\setminus S$ and $S_1$, or $V(C)$ intersects with $V(D)\setminus S$ and $S_2$. We will only show the first situation leads to a contradiction - the second follows similarly. Assume $V(C)$ intersects with $S_1$. Thus, there is an edge $e = uv$ in $C$ such that $u\in V(D)\setminus S$ and $v\in S_1$. But then, by the definition of $m(v)$, the first coordinate of $f(u)$ is not equal to the first coordinate of $f(v)$, contradicting that $C$ is monochromatic.
\end{proof}

Before we can prove Theorem \ref{dichi-main1} we need to introduce some more tools developed in \cite{p4-dichi-bounded}. For a not strongly connected tournament $K$, let $K_1, \dots, K_k$ be the partition of $V(K)$ into its strongly connected components. Let $K^*$ be the tournament that results from contracting each of these parts into a single vertex each. It follows that digraph $K^*$ has vertices $u^*$ and $v^*$ such that $N_{K^*}^-(u)\cap V(K^*)=\emptyset$ and $N_{K^*}^+(v)\cap V(K^*)=\emptyset$. If $u$ is in the component that got contracted to the vertex $u^*$, then we call $u$ a \textit{source vertex}. If $v$ is in the component that got contracted to the vertex $v^*$, then we call $v$ a \textit{sink vertex}. 

We say $C$ is a \textit{path-minimizing closed tournament} (PMCT) if either $V(C)=K$, where $K$ is a strongly connected tournament with $\omega(D)=|K|$, or $V(C)=K\cup V(P)$ where $K$ is a tournament that is not strongly connected, $\omega(D)=|K|$, and $P$ is a directed path from a sink vertex to a source vertex of $K$. Furthermore, $K$ is picked such that $|V(C)| = |V(K)\cup V(P)|$ is minimized. Notice that if $D$ has a strongly connected tournament on $\omega(D)$ vertices, then every PMCT is a tournament. Otherwise, if $C$ is a PMCT, then $C$ is not a tournament, and $K$ is picked such that $|V(P)|$ is as small as possible.

Eventually, we need to go into four different cases. For that, we will illustrate the different cases that we will have. There are 8 types of orientations to consider that we separate into four types. These are illustrated on Figure \ref{type1}, Figure \ref{type2}, Figure \ref{type3}, and Figure \ref{type4}. Since $\mathcal{B}$ and $\mathcal{B'}$ are of opposing orientation, we may assume that $\mathcal{B}$ is of type 1 or type 3, and that $\mathcal{B}$ is of type 2 or type 4, giving four cases.

\begin{figure}
\centering
\begin{subfigure}{.5\textwidth}
  \centering
  \includegraphics[width=.4\linewidth]{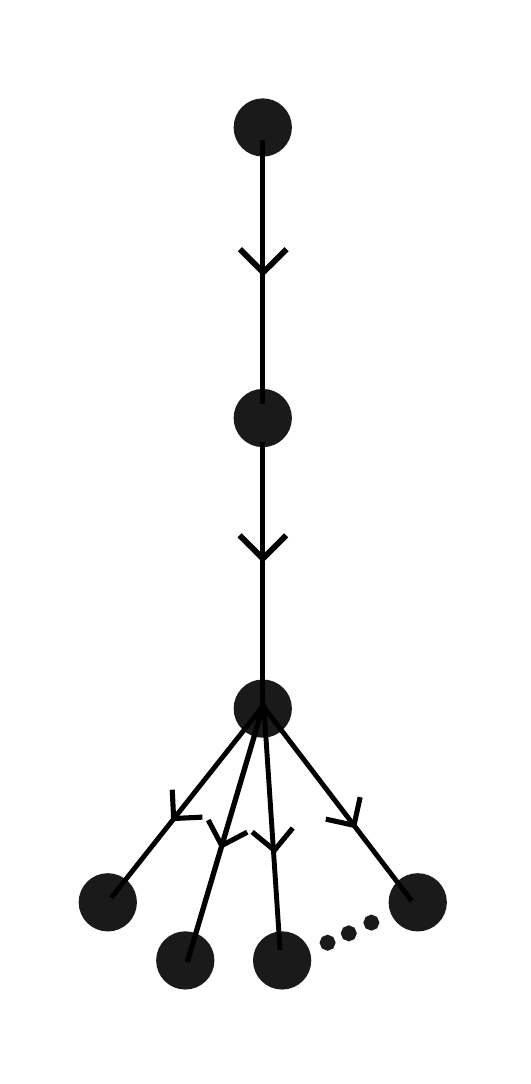}
  \includegraphics[width=.4\linewidth]{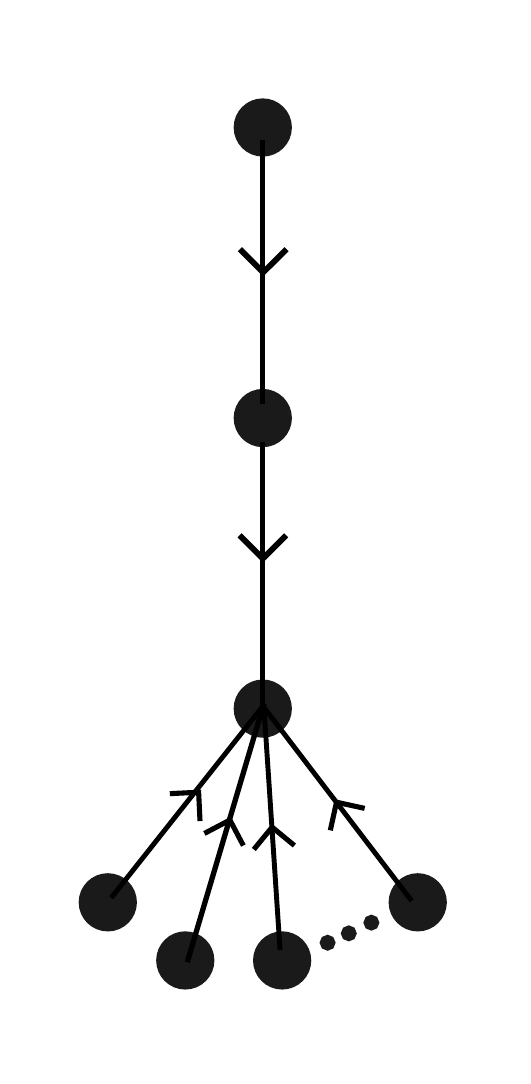}
  \caption{Type 1 brooms.}
  \label{type1}
\end{subfigure}%
\begin{subfigure}{.5\textwidth}
  \centering
  \includegraphics[width=.4\linewidth]{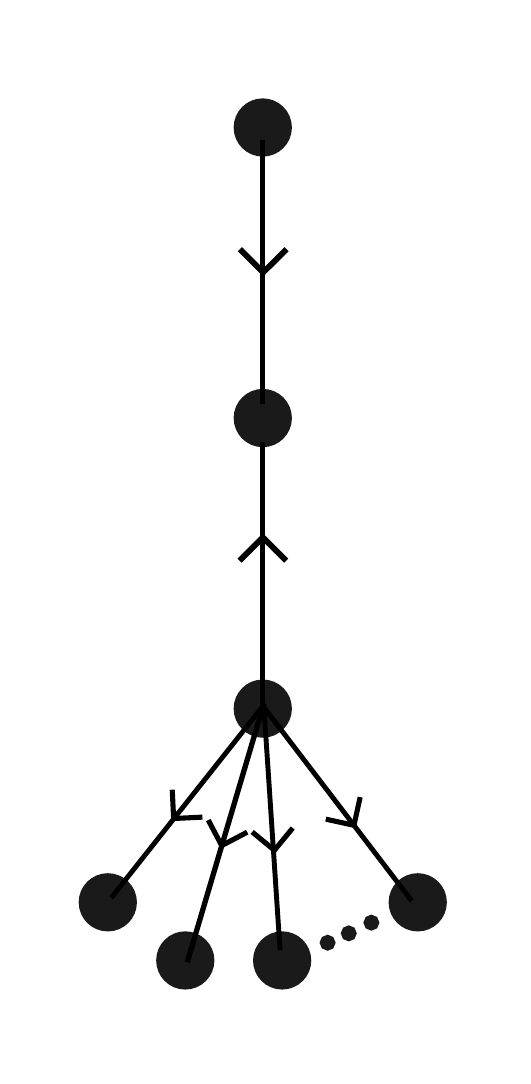}
  \includegraphics[width=.4\linewidth]{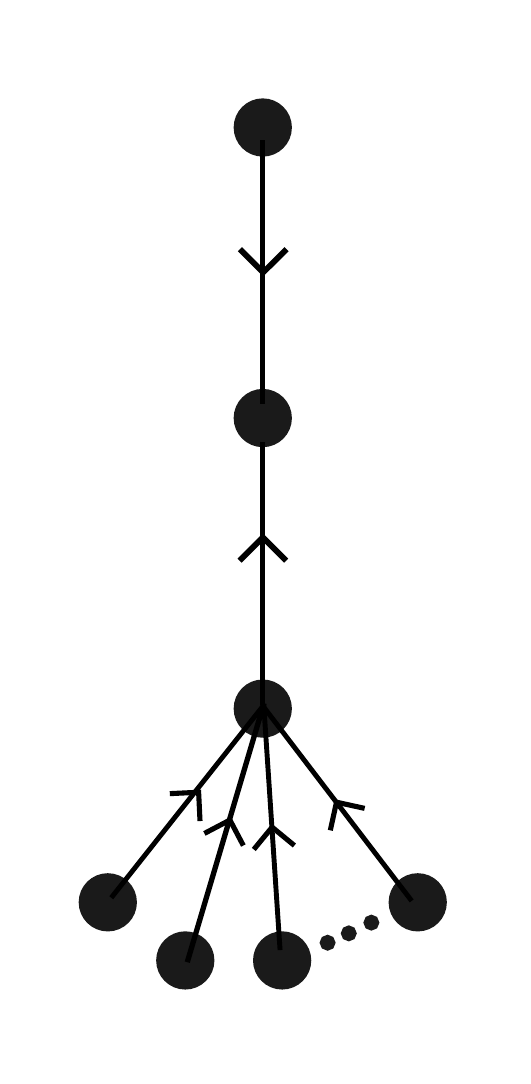}
  \caption{Type 2 brooms.}
  \label{type2}
\end{subfigure}
\label{fig1}
\caption{Type 1 and type 2 brooms.}
\end{figure}


\begin{figure}
\centering
\begin{subfigure}{.5\textwidth}
  \centering
  \includegraphics[width=.4\linewidth]{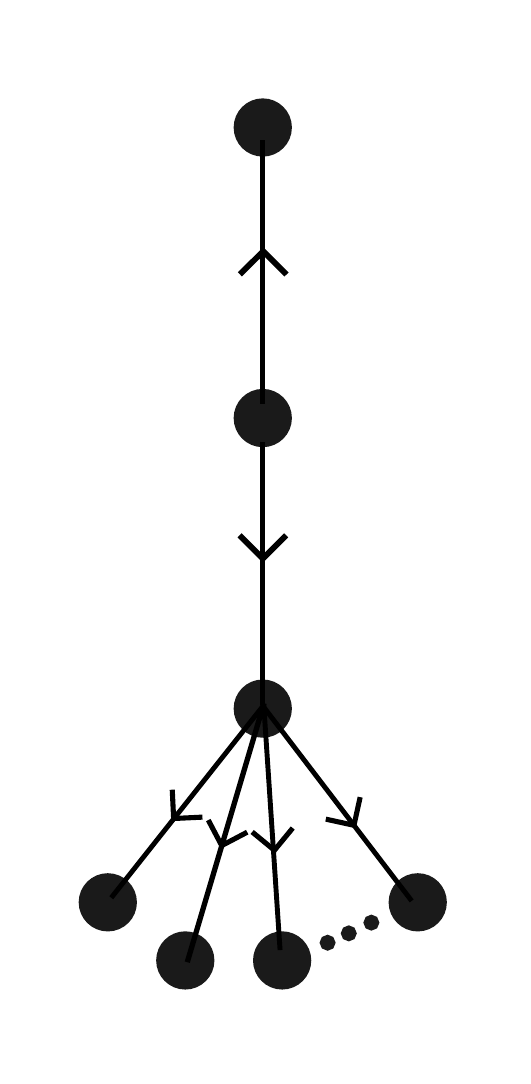}
  \includegraphics[width=.4\linewidth]{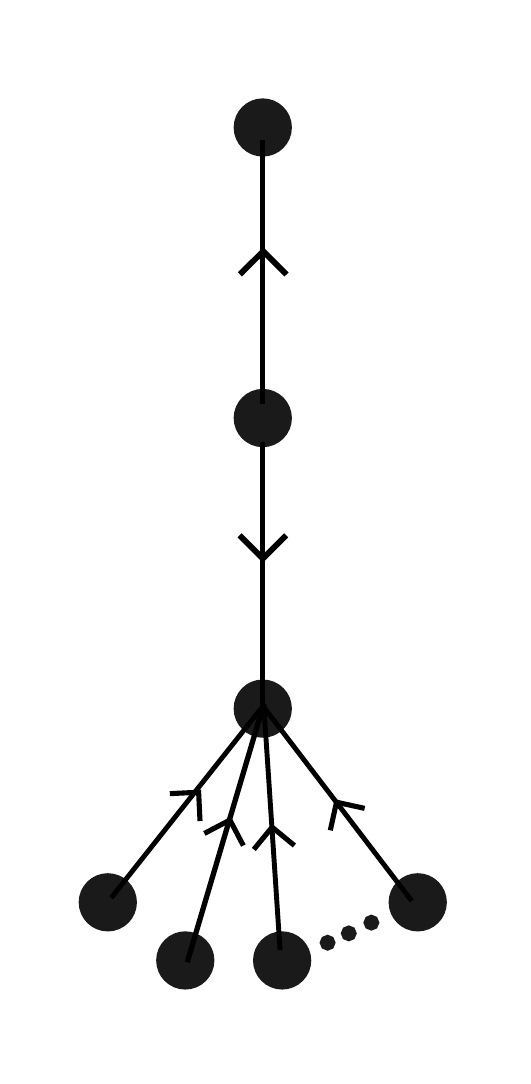}
  \caption{Type 3 brooms.}
  \label{type3}
\end{subfigure}%
\begin{subfigure}{.5\textwidth}
  \centering
  \includegraphics[width=.4\linewidth]{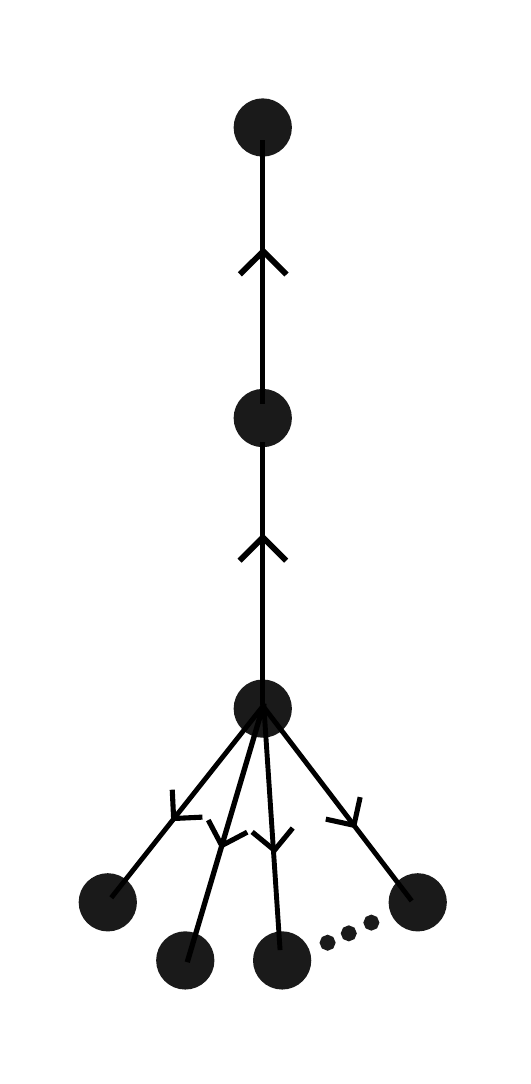}
  \includegraphics[width=.4\linewidth]{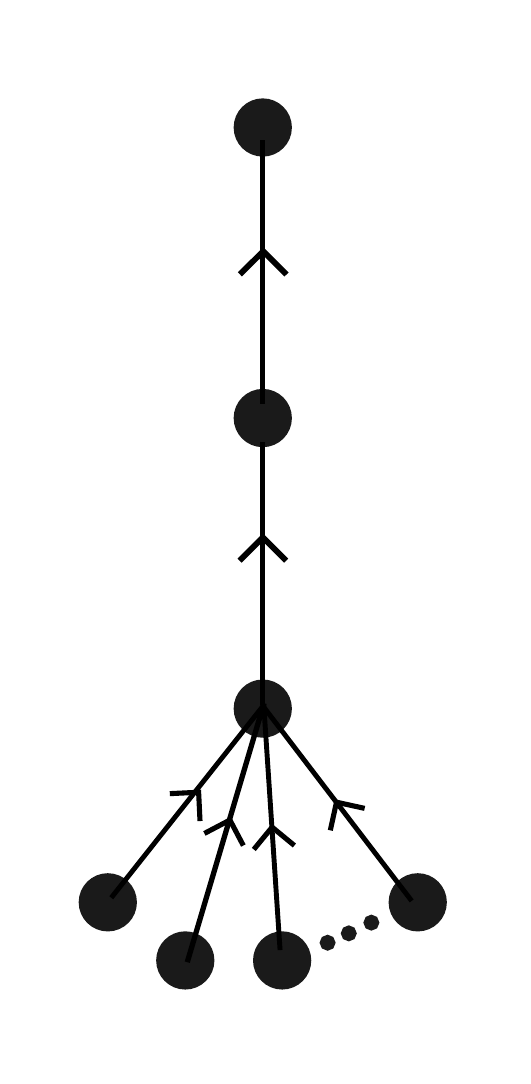}
  \caption{Type 4 brooms.}
  \label{type4}
\end{subfigure}
\label{fig2}
\caption{Type 3 and type 4 brooms.}
\end{figure}

\noindent\textbf{Proof of Theorem \ref{dichi-main1}}: Let $\mathcal{C}$ be the set of $\{\mathcal{B}, \mathcal{B'}\}$-free digraphs. To prove that $\mathcal{C}$ is $\dichi$-bounded, we proceed by induction on $\omega(D)$. The result is immediate if $\omega(D) = 1$. For a digraph $D$, assume that the statement holds for every $\omega<\omega(D)$. That is, assume that there exists a number $\gamma$ such that if $\omega(D')<\omega(D)$ and $D'$ is $\{\mathcal{B}, \mathcal{B}\}$-free, then $\dichi(D')\leq \gamma$. Finally, let $k = \max\{R(r, \omega(D)), R(s, \omega(D)) \}$ where $R$ is the graph Ramsey number. We want to prove that 
$$
\dichi(D)\leq 2(\omega(D)(\gamma+1)+\gamma(6k+25)+2)(k+1).
$$

We may assume that $D$ is strongly connected as the strongly connected components of a digraph can be coloured independently. Let $C$ be a PMCT, which exists since $D$ is strongly connected. Let $X$ be the set of vertices $v\not\in C$ such that $v$ has an in-neighbour and an out-neighbour in $C$, $Z=N(V(C))\setminus X$, and $Y=N(X)\setminus N[V(C)]$.

\sta{\label{dichi-ramsey}
If $S$ is a set of vertices in $D$ such that $|S|\geq k$, then $S$ contains a stable set of size at least $\max\{r,s \}$.
}
\begin{subproof}
    The proof is immediate from the definition of the graph Ramsey number.
\end{subproof}

The following is the analog of the proof of Lemma 3.1 from \cite{p4-dichi-bounded}.

\sta{\label{dichi-nice-set}
$N[C\cup X]$ is a $k$-nice set.
}
\begin{subproof}
We want to prove that if $v\in N[C\cup X]$, then either $v$ has at most $k$ in-neighbours in $V(D)\setminus N[C\cup X]$, or $v$ has at most $k$ out-neighbours in $V(D)\setminus N[C\cup X]$. For this purpose, notice that if $v\in C\cup X$, then the result follows immediately.

For a contradiction, assume that there exists a vertex $v\in N(C\cup X)$ such that $v$ has at least $k$ in-neighbours and out-neighbours not in $N[C\cup X]$. Let $S^-:=N^-(v)\setminus N[C\cup X]$ and $S^+:=N^+(v)\setminus N[C\cup X]$. Either $v\in Y$ or $v\in Z$. If $v\in Y$, then by the definition of $Y$, there exists $x\in X$ such that $x$ is a neighbour of $v$. Since $x\in X$, there exists vertices $c_1, c_2\in C$ such that $c_1x, xc_2\in A(D)$. Note that as $v\in Y$, $v$ is non-adjacent to $c_1, c_2$. Furthermore, notice that $\{x, c_1, c_2 \}$ is anticomplete to $S^-\cup S^+$. Since $\mathcal{B}$ and $\mathcal{B'}$ have opposing orientations, both cases $xv\in A(D)$ and $vx\in A(D)$ each imply that there exists a copy of $\mathcal{B}$ or $\mathcal{B'}$ in $\{c_1, c_2, x, v \}\cup S^-\cup S^+$. Since $D$ is $\{\mathcal{B}, \mathcal{B'} \}$-free, we conclude $v\not\in Y$.

It follows that $v\in Z$. Since $v\not\in X$, $v$ has either only in-neighbours or only out-neighbours in $C$. Furthermore, since $C$ contains a clique of maximal size, $N^0(v)\cap C$ is nonempty. Thus, since $C$ is strongly connected, there is an arc from $N^0(v)\cap C$ to $N(v)\cap C$, and an arc from $N(v)\cap C$ to $N^0(v)\cap C$. Let these arcs be $x_1y_1$ and $y_2x_2$, respectively. Note that $\{x_1, x_2, y_1, y_2 \}$ are anti-complete to $S^+\cup S^-$. As before, both cases where $v$ has only in-neighbours in $C$ or out-neighbours in $C$ imply that the set $\{x_1, x_2, y_1, y_2, v\} \cup S^-\cup S^+$ contain a copy of $\mathcal{B}$ or $\mathcal{B'}$. Since $D$ is $\{\mathcal{B}, \mathcal{B'} \}$-free, both lead to contradictions. We conclude $N[C\cup X]$ is a $k$-nice set.
\end{subproof}

By using Lemma \ref{k-nice-sets}, it is enough to bound $\dichi(N[C\cup X])$. If $C$ is a strongly connected tournament, then we consider $P$ to be the empty path. As noted by Cook, Masařík, Pilipczuk, Reinald, and Souza \cite{p4-dichi-bounded},
$$
\dichi(N[C\cup X])\leq \dichi(N[K]) + \dichi(P)+\dichi(N(P)\setminus N[K]) + \dichi(Y).
$$

\begin{figure}
    \centering
    \includegraphics{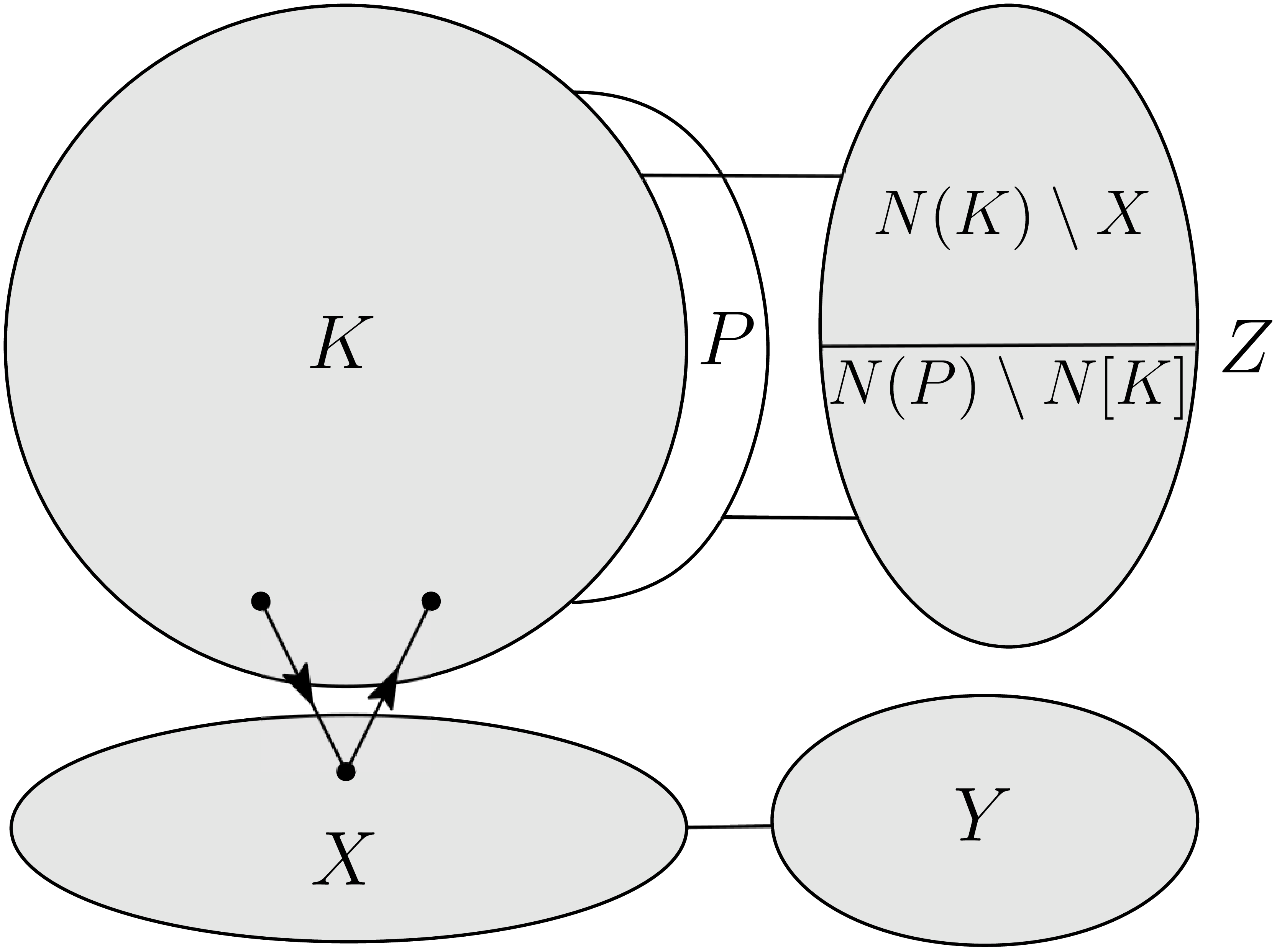}
    \caption{An illustration of $N[C\cup X]$.}
    \label{dichi-figures-nice}
\end{figure}

For an illustration of $N[C\cup X]$, see Figure \ref{dichi-figures-nice}. Thus, we want to bound each of these. By the minimality of $|V(P)|$ and by Observation 4.1 in \cite{p4-dichi-bounded}, we have $\dichi(P)\leq 2$. Furthermore, since $\dichi(N(v))\leq \gamma$ for every $v\in V(D)$ by the definition of $\gamma$, we have $\dichi(N[K])\leq \omega(D)+\omega(D)\gamma = \omega(D)(\gamma+1)$. We proceed to bound $\dichi(Y)$. The following is the analog of Lemma 4.3 and Corollary 4.4 in \cite{p4-dichi-bounded}. However, we use $k$-nice sets to get brooms rather than paths.

\sta{\label{dichi-bound-1}
$\dichi(Y)\leq 2\gamma(k+1)$.
}
\begin{subproof}
    We proceed by proving that every non-empty induced subgraph $Y'$ of $Y$ has a $k$-nice set $S$ such that $\dichi(S)\leq \gamma$, which finishes the proof by Lemma \ref{k-nice-sets}. The statement is true for $Y'=\emptyset$, so we may assume $Y'$ is not empty. 

    By the definition of $Y$, there exists a vertex $x\in X$ such that $N(x)\cap Y'\not=\emptyset$. By the definition of $X$, there exists vertices $c_1, c_2\in C$ such that $c_1x, xc_2\in A(D)$. Set $S = N(x)\cap Y'$. By the definition of $\gamma$, we get $\dichi(S)\leq \gamma$. It suffices to prove that $S$ is a $k$-nice set. For a contradiction, assume that there exists a vertex $s\in S$ which has at least $k$ in-neighbours and out-neighbours in $Y'\setminus N(x)$. Let $S^-:=N^-(s)\cap (Y'\setminus N(x))$ and $S^+:=N^+(s)\cap (Y'\setminus N(X))$. If $xs\in A(D)$, then $\{c_1, x, s \}\cup S^+$ induces a subgraph that contains a copy of $\mathcal{B'}$ by Claim (\ref{dichi-ramsey}). A similar argument works if $sx\in A(D)$. This proves that $S$ is a $k$-nice set, thus finishing the proof.
\end{subproof}

It remains to bound $\dichi(N(P)\setminus N[K])$. How we do it mimics the technique used in Section 5 of \cite{p4-dichi-bounded}. If $C$ is a strongly connected tournament, then $P$ is the empty path, so $\dichi(N(P))=0$. Assume then that $K$ is not strongly connected. If $P$ has at most four vertices, then $\dichi(N(P))\setminus N[K]\leq 4\gamma$. Assume then that $P$ has more than four vertices. Let $P'$ be the path $P$ with the first, the second first, the last, and the second-to-last vertices deleted. Let $Q$ be the set of these four vertices. We want to bound the dichromatic number of $N(P')\setminus (N[K]\cup N(Q))$.

Let $v_1,\dots, v_n$ denote the vertices of $P'$ labeled such that $v_iv_{i+1}\in A(D)$ for every $i\in \{1,\dots,m-1\}$. When talking about vertices $v$ in $N(P')$, we use \textit{first (in/out-)neighbour} of $v$ to refer to the vertex $v_i$ that is a (in/out-)neighbour of $v$ such that $i$ is minimized. Similarly, the \textit{last (in/out-)neighbour} of $v$ refers to the vertex $v_i$ that is a (in/out-)neighbour of $v$ such that $i$ is maximized. 

The rest of the proof is by cases. Notice that since $\mathcal{B}$ and $\mathcal{B'}$ have opposing consistent orientations, we may assume without loss of generality that $\mathcal{B}$ is of type 1 or type 3, and $\mathcal{B'}$ is of type 2 or type 4.

Let $A^-$ (resp $A^+$) be the set of vertices in $N(P')\setminus (N[K]\cup N(Q))$ such that their first neighbour is an in-neighbour (resp out-neighbour). Furthermore, let $B^+$ (resp. $B^-$) be the set of vertices in $N(P')\setminus (N[K]\cup N(Q))$ such that their last neighbour is an out-neighbour (resp. in-neighbour). The following claim will make the proofs of each case less repetitive. 

\sta{ \label{broom-types}
 The following are true.
}
\begin{itemize}
    \item \textit{If $\mathcal{B}$ is of type 1, then $\dichi(A^-)\leq 2\gamma(k+1)$.}

    \item \textit{If $\mathcal{B'}$ is of type 2, then $\dichi(A^+)\leq 2\gamma(k+1)$.}

    \item \textit{If $\mathcal{B}$ is of type 3, then $\dichi(B^-)\leq 2\gamma(k+1)$.}

    \item \textit{If $\mathcal{B'}$ is of type 4, then $\dichi(B^+)\leq 2\gamma(k+1)$.}
\end{itemize}

\begin{subproof}
    Let us prove the first bullet point. We will use Lemma \ref{k-nice-sets} to bound $\dichi(A^-)$. Let $i$ be the smallest integer such that $N^+(v_i)\cap A^-\not = \emptyset$. By the definition of $\gamma$, we have $\dichi(N^+(v_i)\cap A^-)\leq \gamma$. We claim $N^+(v_i)\cap A^-$ is a $k$-nice set in $D[A^-]$. Let $v\in N^+(v_i)\cap A^-$ be such that $v$ has at least $k$ in-neighbours and at least $k$ out-neighbours not in $A^-\setminus N^+(v_i)$. Let $S^+:=N^+(v)\cap (A^-\setminus N^+(v_i))$ be the set of out-neighbours of $v$ in $A^-$, and let $S^-:=N^-(v)\cap (A^-\setminus N^+(v_i))$ be the set of out-neighbours of $v$ in $A^-\setminus N^+(v_i)$. 
    
    Since vertices in $A^-$ have $v_i$ as their first neighbour, $v_{i-1}$, where we pick $v_{i-1}$ as the second vertex of $P$ if $i=1$, is anticomplete to $S^+\cup S^-$. Furthermore, since vertices in $S^+\cup S^-$ are not in $N^+(v_i)$, and $S^+\cup S^-\subseteq A^-$, we have that $v_i$ is anticomplete to $S^-\cup S^+$. 

    This, however, implies that $\{v_{i-1}, v_i, v\}\cup S^+\cup S^-$ contains a copy of $\mathcal{B}$. We conclude that every vertex in $N^+(v_i)\cap A^-$ contains at most $k$ out-neighbours in $A^-\setminus N^+(v_i)$. This finishes the proof that $N^+(v_i)\cap A^-$ is a $k$-nice set in $D[A^-]$, and so the proof that $\dichi(A^-)\leq 2\gamma(k+1)$. 
    
    Similar arguments prove the remaining bullet points.
\end{subproof}

In the following sections, we will prove that, for every case, $\dichi(N(P')\setminus (N[K]\cup N(Q)))\leq \gamma(4k+19)$. This will finish the proof since then,
\begin{align*}
    \dichi(N[K\cup X]) & \leq \dichi(N[K]) + \dichi(P)+\dichi(N(P)\setminus N[k]) + \dichi(Y) \\
    & \leq \omega(D)(\gamma+1)+2+4\gamma+\gamma(4k+19)+\gamma(2k+2) \\
    & = \omega(D)(\gamma+1)+\gamma(6k+25)+2,
\end{align*}
where the $4\gamma$ in the second line came from $N(Q)$, the neighbourhood of the vertices in $P$ that are not in $P'$. By Lemma \ref{k-nice-sets},
$$
\dichi(D)\leq 2(\omega(D)(\gamma+1)+\gamma(6k+25)+2)(k+1),
$$
as claimed.

\subsection{Brooms of type 1 and type 2}

Assume that $\mathcal{B}$ is a broom of type 1, and that $\mathcal{B'}$ is a broom of type 2. By (\ref{broom-types}), $\dichi(A^-\cup A^+)\leq 4\gamma(k+1)$. However, $A^-, A^+$ partitions $N(P')\setminus (N[K]\cup N(Q))$, so $\dichi(N(P')\setminus (N[K]\setminus N(Q)))\leq \gamma(4k+4)\leq \gamma(4k+19),$ as claimed.

\subsection{Brooms of type 3 and type 4}

Assume that $\mathcal{B}$ is a broom of type 3, and that $\mathcal{B'}$ is a broom of type 4. By (\ref{broom-types}), $\dichi(B^-\cup B^+)\leq 4\gamma(k+1)$. However, $B^-, B^+$ partitions $N(P')\setminus (N[K]\setminus N(Q))$, so $\dichi(N(P')\setminus (N[K]\setminus N(Q)))\leq \gamma(4k+4)\leq \gamma(4k+19)$, as claimed.

\subsection{Brooms of type 2 and type 3}

Assume that $\mathcal{B}$ is a broom of type 3, and that $\mathcal{B'}$ is a broom of type 2. By (\ref{broom-types}), $\dichi(A^+\cup B^-)\leq 4\gamma(k+1)$. Let $C'=N(P')\setminus (N[K]\cup N(Q)\cup A^+\cup B^-)$, and for every $i\in \{1,\dots,m\}$, let $L_i$ be the set of vertices $v$ in $C'$ such that $v_iv\in A(D)$ and $i$ is minimized. 

\sta{\label{dichi-bound-5}
For every $i\geq 1$ and $j\geq i+3$, if $v\in L_i$, then $N^+(v)\cap L_j=\emptyset$.
}
\begin{subproof}
Assume for a contradiction that there exists a vertex $u\in N^+(v)\cap L_j$. Since $u\not\in A^+$, $v_j$ is the first neighbour of $v$ in $P$. Since $u\not\in B^-$, $u$ has an out-neighbour $v_l$ in $P$ with $l>j$. But then we can shorten the path $P$ by replacing vertices $v_{i+1},\dots,v_{l-1}$ with $v$ and $u$. This contradicts $C$ is a PMCT.  
\end{subproof}
Set
$$
C_i=\{c\in L_j\text{ : }i\text{ is congruent to }j\text{ modulo }3\}.
$$
That is, for example, $C_1$ is the union of $L_1, L_4, L_7$, and so on. Furthermore, by (\ref{dichi-bound-5}), every strongly connected subdigraph $D'$ of $C_i$ is contained in a set $L_i$. Since $L_i$ is contained in the neighbourhood of $v_i$, it follows that $\dichi(L_i)\leq \gamma$. Thus, $\dichi(C_i)\leq \gamma$.

We can now finish the proof. Since $C_1, C_2, C_3$ is a partition of $C'$, it follows that $\dichi(C')\leq 3\gamma$. Thus, 
\begin{align*}
    \dichi(N(P')\setminus N[K])& \leq \dichi(A^-\cup B^-) + \dichi(C') \\
    & \leq 4\gamma(k+1) + 3\gamma \\
    & \leq \gamma(4k+7) \\
    & \leq \gamma(4k+19).
\end{align*}
as claimed.

\subsection{Brooms of type 1 and type 4}

Assume that $\mathcal{B}$ is a broom of type 1, and that $\mathcal{B'}$ is a broom of type 4. By (\ref{broom-types}), $\dichi(A^-\cup B^+)\leq 4\gamma(k+1)$. Let $C'=N(P')\setminus (N[K]\cup N(Q)\cup A^-\cup B^+)$. $C'$ does not contain a strongly connected tournament on $\omega(D)$ vertices by the minimality of $|V(P)|$. For every $i\in\{ 1,\dots,m\}$, let $L_i$ be the set of vertices $v$ in $C'$ such that $vv_i\in A(D)$ and $i$ is minimized. Notice that since $v\not\in A^-\cup B^+$, it follows that $v$ has both an in-neighbour and out-neighbour in $P'$, and so $L_1, \dots, L_m$ partitions $C'$. Finally, let $C_1, \dots, C_5$ be such that:
$$
C_i = \{c\in L_j\text{ : }i\text{ is congruent to }j\text{ modulo }5 \}.
$$
That is, for example, $C_1$ is the union of $L_1, L_6, L_{11}$, and so on. We will bound $\dichi(C_i)$ by partitioning each of $C_1,\dots,C_5$ into three sets each with a clique number strictly smaller than $\omega(D)$. This will imply that $\dichi(C')\leq 15\gamma$. The following claim will allow us to make such a partition.

\sta{\label{dichi-bound-3}
Let $1\leq i\leq 5$, and let $v\in C_i$. If $K_1$ and $K_2$ are tournaments in $C_i$ each of size $\omega(D)$, then $v$ is not both a sink vertex of $K_1$ and a source vertex of $K_2$.
}
\begin{subproof}
    For a contradiction, assume the claim does not hold. That is, suppose there exists a vertex $v$ and tournaments $K_1$ and $K_2$ each of size $\omega(D)$ such that $v$ is a source vertex of $K_1$ and a sink vertex of $K_2$. Let $u$ be a sink vertex in $K_1$ and $w$ be a source vertex in $K_2$. Note that this implies that $wv,vu\in A(D)$. Let $v_i$ and $v_j$ be the first out-neighbour and in-neighbour of $v$ respectively. Furthermore, let $v_x$ be the first in-neighbour of $w$, and let $v_y$ be the last out-neighbour of $u$. Since $v\not\in A^-\cup B^+$, we have $i<j$.  Furthermore, if $i\leq x$, then $K_1$ and $v_i\rightarrow \dots \rightarrow v_x$ contradict the minimality of $C$ as a PMCT. Thus, $x<i$. Using similar logic, we also have $j<y$. 
    
    By the definition of $C_i$, we have $x\cong j\mod{5}$, and since $x<j$, we have that $|x-j|\geq 5$. Consequently, the path $P''$ which is $P'$ with vertices $v_x, \dots, v_y$ replaced by $v_x, w, v, u, v_y$, is strictly smaller. This contradicts the minimality of $C$, thus finishing the proof.
\end{subproof}

\sta{\label{dichi-bound-4}
For every $1\leq i\leq 5$, we have $\dichi(C_i)\leq 3\gamma$.
}
\begin{subproof}
    Let $X_i$ (resp. $Y_i$) be the set of vertices $v\in N(P')\setminus (N[K]\cup N(Q))$ such that there exists a tournament $K'$ with $|K'|=\omega(D)$ in $C_i$ where $v$ is a sink vertex (resp. source vertex) of $K'$, and let $Z_i = C_i\setminus (X_i\cup Y_i)$. If $\omega(X_i) = \omega(D)$, then there exists a tournament $K'$ in $X_i$ with a source vertex $v$. But since $v\in X_i$, then $v$ is a sink vertex of another tournament, this contradicts (\ref{dichi-bound-3}). Thus, $\omega(X_i)<\omega(D)$. By similar logic, $\omega(Y_i)<\omega(D)$. As for $Z_i$, each $\omega(D)$-vertex tournament in $Z_i$ is strongly connected by the choice of $X_i$ and $Y_i$. But since $P\not=\emptyset$, this contradicts that $C$ is a PMCT. Thus, $\omega(Z_i)<\omega(D)$. We conclude $\dichi(C_i)\leq 3\gamma$.   
\end{subproof}

We can now finish the proof. Since $C_1, \dots, C_5$ is a partition of $C'$, it follows that $\dichi(C')\leq 15\gamma$. Thus, 
\begin{align*}
    \dichi(N(P')\setminus (N[K]\cup N(Q)))& \leq \dichi(A^-\cup B^+) + \dichi(C') \\
    & \leq 4\gamma(k+1) + 15\gamma \\
    & \leq \gamma(4k+19).
\end{align*}
as claimed, which finishes the proof of this case, and so the proof of Theorem \ref{dichi-main1}. \qed{}

\bibliographystyle{plain}
\bibliography{thebib.bib}

\begin{thebibliography}{10}

\bibitem{complete-multipartite}
Pierre Aboulker, Guillaume Aubian, and Pierre Charbit.
\newblock Heroes in oriented complete multipartite graphs.
\newblock {\em arXiv preprint arXiv:2202.13306}, 2022.

\bibitem{gyarfas-sumner-digraphs}
Pierre Aboulker, Pierre Charbit, and Reza Naserasr.
\newblock Extension of {G}y{\'a}rf{\'a}s-{S}umner conjecture to digraphs.
\newblock {\em Electronic Journal of Combinatorics}, 2021.

\bibitem{heroes-characterization}
Eli Berger, Krzysztof Choromanski, Maria Chudnovsky, Jacob Fox, Martin Loebl,
  Alex Scott, Paul Seymour, and Stephan Thomass{\'e}.
\newblock Tournaments and colouring.
\newblock {\em Journal of Combinatorial Theory, Series B}, 103(1):1--20, 2013.

\bibitem{stars-dichi-bounded}
Maria Chudnovsky, Alex Scott, and Paul Seymour.
\newblock Induced subgraphs of graphs with large chromatic number. {XI}.
  {O}rientations.
\newblock {\em European Journal of Combinatorics}, 76:53--61, 2019.

\bibitem{p4-dichi-bounded}
Linda Cook, Tom{\'a}{\v{s}} Masa{\v{r}}{\'\i}k, Marcin Pilipczuk, Amadeus
  Reinald, and U{\'e}verton~S Souza.
\newblock Proving a directed analogue of the {G}y{\'a}rf{\'a}s-{S}umner
  conjecture for orientations of $ {P}_4$.
\newblock {\em arXiv preprint arXiv:2209.06171}, 2022.

\bibitem{shift-graphs}
Paul Erd{\H{o}}s and Andr{\'a}s Hajnal.
\newblock On chromatic number of graphs and set-systems.
\newblock {\em Acta Math. Acad. Sci. Hungar}, 17(61-99):1, 1966.

\bibitem{gallai-roy-1}
Tibor Gallai.
\newblock On directed paths and circuits.
\newblock {\em Theory of Graphs}, pages 115--118, 1968.

\bibitem{gyarfas-sumner-1}
Andr{\'a}s Gy{\'a}rf{\'a}s.
\newblock {\em Problems from the world surrounding perfect graphs}.
\newblock Number 177. MTA Sz{\'a}m{\'\i}t{\'a}stechnikai {\'e}s
  Automatiz{\'a}l{\'a}si Kutat{\'o} Int{\'e}zet, 1985.

\bibitem{dense-digraphs}
Ararat Harutyunyan, Tien-Nam Le, Alantha Newman, and St{\'e}phan Thomass{\'e}.
\newblock Coloring dense digraphs.
\newblock {\em Combinatorica}, 39:1021--1053, 2019.

\bibitem{first-dicoloring}
Victor Neumann-Lara.
\newblock The dichromatic number of a digraph.
\newblock {\em Journal of Combinatorial Theory, Series B}, 33(3):265--270,
  1982.

\bibitem{gallai-roy-2}
Bernard Roy.
\newblock Nombre chromatique et plus longs chemins d'un graphe.
\newblock {\em Revue fran{\c{c}}aise d'informatique et de recherche
  op{\'e}rationnelle}, 1(5):129--132, 1967.

\bibitem{chi-survey}
Alex Scott and Paul Seymour.
\newblock A survey of $\chi$-boundedness.
\newblock {\em Journal of Graph Theory}, 95(3):473--504, 2020.

\bibitem{steiner2022coloring}
Raphael Steiner.
\newblock On coloring digraphs with forbidden induced subgraphs.
\newblock {\em Journal of Graph Theory}, 103(2):323--339, 2023.

\bibitem{gyarfas-sumner-2}
David~P Sumner.
\newblock Subtrees of a graph and chromatic number.
\newblock {\em The Theory and Applications of Graphs,(G. Chartrand, ed.), John
  Wiley \& Sons, New York}, 557:576, 1981.

\end{thebibliography}

\end{document}